\documentclass[letterpaper]{article}
\pdfoutput=1
\usepackage{natbib}
\usepackage[english]{babel}
\usepackage[utf8x]{inputenc}
\usepackage{amsmath}
\usepackage{graphicx}
\usepackage{hhline}
\usepackage[usenames,dvipsnames,svgnames]{xcolor} 
\usepackage[colorinlistoftodos]{todonotes} 
\usepackage{algpseudocode}
\usepackage{pifont}
\usepackage{amsmath}
\usepackage[boxruled,vlined,linesnumbered]{algorithm2e}
\usepackage{color}
\usepackage{amsthm}
\usepackage[multiple]{footmisc}
\usepackage{mathtools}
\mathtoolsset{showonlyrefs}

\usepackage{etex}

\usepackage{floatrow}
\newfloatcommand{capbtabbox}{table}[][\FBwidth]

\usepackage{array}
\newcolumntype{C}[1]{>{\centering\let\newline\\\arraybackslash\hspace{0pt}}m{#1}}

\usepackage[margin=1in,ignorefoot]{geometry} 

\usepackage{caption}
\usepackage{enumitem}
\usepackage{hhline}
\usepackage{tikz}
\usepackage{amssymb,amsthm}
\usepackage[T1]{fontenc}
\usepackage{subcaption}
\usepackage{lmodern} 
\usepackage{hyperref} 
\hypersetup{pdfauthor={Aleksandr Aravkin, Stephen Becker, Derek Driggs},
    pdftitle={},
    colorlinks=true,
    citecolor=MidnightBlue,
    urlcolor=Bittersweet,
}

\usepackage{slashbox} 
\usepackage{booktabs}

\newcommand{\X}{X} 
\newcommand{\defeq}{\stackrel{\text{\tiny def}}{=}}

\newcommand{\rankVar}{k} 
\newcommand{\dimVar}{n}
\newcommand{\dimVarX}{\dimVar \times \dimVar}
\renewcommand{\P}{P}

\usepackage{pgfplots}
\usepgfplotslibrary{dateplot}

\newcommand{\R}{\mathbb{R}}
\DeclareMathOperator{\rank}{rank}

\DeclareMathOperator{\trace}{trace}

\DeclareMathOperator*{\argmin}{arg\,min}

\newcommand{\<}{\langle}
\renewcommand{\>}{\rangle}

\newtheorem{theorem}{Theorem}
\newtheorem{remark}{Remark}
\newtheorem{lemma}{Lemma}

\newtheoremstyle{dotless}{}{}{\itshape}{}{\bfseries}{}{ }{}
\theoremstyle{dotless}
\newtheorem*{proposition*}{Proposition}

\title{Adapting Regularized Low-Rank Models for Parallel Architectures}
\author{Derek Driggs\thanks{University of Colorado, Boulder CO 80309 USA}, Stephen Becker\footnotemark[1], Aleksandr Aravkin\thanks{University of Washington, Seattle, WA 98195 USA}}

\begin{document}
\maketitle

\begin{abstract}

We introduce a reformulation of regularized low-rank recovery models to take advantage of GPU, multiple CPU, and hybridized architectures. Low-rank recovery often involves nuclear-norm minimization through iterative thresholding of singular values. These models are slow to fit and difficult to parallelize because of their dependence on computing a singular value decomposition at each iteration. Regularized low-rank recovery models also incorporate non-smooth terms to separate structured components (e.g. sparse outliers) from the low-rank component, making these problems 
more difficult. 

Using Burer-Monteiro splitting and marginalization, we develop a smooth, non-convex formulation of regularized low-rank recovery models that can be fit with first-order solvers. We develop a computable certificate of convergence for this non-convex program, and use it to establish bounds on the suboptimality of any point. Using robust principal component analysis (RPCA) as an example, we include numerical experiments showing that this approach is an order-of-magnitude faster than existing RPCA solvers on the GPU. We also show that this acceleration allows new applications for RPCA, including real-time background subtraction and MR image analysis.

\end{abstract}

\section{Introduction}
\label{sec:intro}
Low-rank matrix decompositions are an effective tool for large-scale data analytics. Background subtraction, facial recognition, document indexing, and collaborative filtering all use 
low-rank matrix recovery algorithms~\citep{RPCA}. However, a low-rank structure is only a part of the story --- signals often exhibit additional structure. \textit{Regularized} low-rank recovery models seek to decompose a matrix $X \in \R^{m \times n}$ into the sum of a low-rank component $L \in \R^{m \times n}$ and another structured matrix $S \in \R^{m \times n}$, where $S$ could be sparse, clustered, non-negative, or have some other useful attribute (see e.g. \cite{GLRM}).

Searching for sparse and low-rank models gives rise to combinatorial formulations, which are intractable. Fortunately, relaxations often work well, and are amenable to 
modern optimization techniques. Consider finding an approximate decomposition of a data matrix $X$ into a sum $X = L + S$, where $S$ is sparse and $L$ low-rank. We are given partial linear measurements $b = \mathcal{A}(X)$. We can formulate the problem   
\begin{equation}
\label{eq:comb}
\min_{L,S} \quad \rank(L) + \textnormal{card}( S ), \quad \textnormal{subject to:} \quad \frac{1}{2} \| \mathcal{A}(L + S) -b\|_2^2 \le \epsilon,
\end{equation}
where the cardinality function $\textnormal{card}( \cdot )$ counts the number of nonzero entries of $S$, and the parameter $\epsilon$ \textcolor{black}{accounts for noise or model inconsistencies}. 
Rank is analogous to sparsity, since it is precisely the $\textnormal{card}( \cdot )$ function applied to the singular values of the argument. 
Problem~\eqref{eq:comb} is NP-hard~\citep{NatarajanSparseNP} and intractable for large problem sizes. However, both terms admit non-smooth convex relaxations.
 $\textnormal{card}( \cdot )$ can be replaced with the $\ell_1$-norm, and the rank function can be replaced with the nuclear norm, $\| \cdot \|_*$, 
 which is equal to the sum of the singular values of the argument:
\begin{equation}
\label{eq:relax}
\min_{L,S} \quad \| L \|_* + \| S \|_1, \quad \textnormal{subject to:} \quad \frac{1}{2} \| \mathcal{A}(L + S) - b \|_2^2 \le \epsilon
\end{equation}
\textcolor{black}{When the parameter $\epsilon$ is set to zero,~\eqref{eq:relax} is called {\it robust PCA} (RPCA), and can recover the lowest-rank $L$ and the sparsest $S$ under mild conditions. Setting $\epsilon>0$ in~\eqref{eq:relax} is more suitable for most applications, since it allows for noise in the measurements}; the problem is then called {\it stable principal component pursuit} (SPCP) \citep{RPCA_algo_Wright,StablePCP}.
SPCP stably recovers $L$ and $S$ with error bounded by $\epsilon$ \citep{StablePCP}. 
Both RPCA and SPCP are convex problems, much easier to solve than the original problem~\eqref{eq:comb}. 
While~\eqref{eq:relax} is an important problem, it is just one example of regularized low-rank models. 

\paragraph{Problem class.}
We are interested in general regularized low-rank approaches:
\begin{equation}
\label{eq:gen}
\min_{L,S} \quad \| L \|_* + \mathcal{L}( b, \mathcal{A}(L + S) ) + r(S),
\end{equation}
where $\mathcal{L}$ is a loss function and $r$ is a convex regularizer. $\mathcal{L}$, which can be {\it infinite-valued}, 
measures how well $L+S$ agrees with measured values of $X$. In RPCA, $\mathcal{L}(v) = \delta_{b}(v)$ is the
indicator function ensuring we match observed data, while in SPCP, $\mathcal{L}(v) = \delta_{2\sqrt{\epsilon}\mathbb{B}}(v-b)$ ensures
we are close to $b$. 
$\mathcal{L}$ can also be finite valued --- for example, a modified SPCP is given by solving
\begin{equation}
\label{eq:lag}
\min_{L,S} \quad \lambda_L \| L \|_* + \frac{1}{2} \| \mathcal{A}(L + S) -b \|_2^2 + \lambda_S \| S \|_1,
\end{equation}
where $\lambda_L, \lambda_S$ are tuning parameters, and $\mathcal{L}(\cdot) = \frac{1}{2}\|\cdot\|_2^2$. 

 The regularizer $r(S)$ is a convex penalty that promotes the desired structure of $S$. The $\ell_1$ norm $r(S) = \| S \|_1$ promotes sparsity in $S$;
 other important examples include the ordered-weighted $\ell_1$ norm (OWL norm) \cite{SLOPE,davis2015,OWLnorm}, 
 and the elastic net \citep{ElasticNet}, which both enforce sparsely correlated observations (clusters) in $S$. 
Other convex regularizers promote known group-sparse and hierarchical structures in $S$; 
see \cite{BachPenalties} for a survey.

\paragraph{Marginalization in $S$.} 
In all of the applications we consider, $r(S)$ is simple enough that we can efficiently minimize~\eqref{eq:gen} in $S$ for a given $L$. 
In the motivating SPCP example~\eqref{eq:lag}, we have 
\[
\min_S \left\{  \lambda_L \| L \|_* + \frac{1}{2} \| \mathcal{A}(L + S) -b \|_2^2 + \lambda_S \| S \|_1 \right\} = \lambda_L \| L \|_* + \rho_{\lambda_S}(\mathcal{A}(L) - b),
\]
where $\rho_{\lambda_S}$ is the Huber function (for a detailed derivation, see e.g.~\cite{LevelSet}). 
This motivates the definition
\begin{equation}
\varphi: L \mapsto  \min_{S} \left\{\mathcal{L}( \mathcal{A}(L+S) -b  ) + \lambda_S \cdot r(S)\right\},
\label{PhiDef}
\end{equation}
where $\varphi$ is smooth (see Section~\ref{Smooth}) with value and derivative 
either explicitly available as in the case of SPCP, or efficiently computable. 
After marginalizing out $S$,~\eqref{eq:gen} becomes 
\begin{equation}
\label{eq:genConvex}
\min_L\; \lambda_L \| L \|_* + \varphi(L).
\end{equation}
 Proximal gradient methods are applicable, 
but require computing the expensive and communication-heavy singular value decomposition (SVD) at each iteration. 
To ameliorate this bottleneck, it is possible to parallelize the SVD step using communication-avoiding algorithms. We develop such an approach in Section \ref{tsqr} to use as a performance baseline.

However, the focus of this paper is to use a non-convex, factorized representation of the nuclear norm that avoids the SVD entirely. This reformulated program offers more speedup through parallelization than communication-avoiding approaches, and the introduced non-convexity can be easily controlled. The smooth structure of $\varphi$ allows a computable {\it optimality certificate} for the non-convex case, which is developed and 
explored in Section~\ref{TheCon}.

\paragraph{Non-convex reformulation.}
To avoid the SVD entirely, we consider non-convex reformulations for regularized low-rank recovery models, 
using the factorization $L = UV^T$, where $U \in \R^{m \times k}, \ V \in \R^{n \times k}$, and $k \ll m,n$. This factorization imposes the rank-constraint $\mathrm{rank}(L) \le k$. In addition, we can further penalize rank using the representation (see e.g.~\cite{BachSparseFactor,Recht07,MMMF2005})
\begin{equation}
\label{eq:factChar}
\| L \|_* \equiv \inf_{L = UV^T} \,\frac{1}{2} ( \| U \|_F^2 + \| V \|_F^2 ).  
\end{equation}
This equivalence only holds if the ranks of $U$ and $V$ are large enough, and makes it possible to  
 maintain a low-rank $L = UV^T$ without requiring SVDs. 
The non-convex analogue to~\eqref{eq:genConvex} is given by 
\begin{equation}
\label{eq:smoothNC}
\argmin_{U,V} \quad \varphi(UV^T) +  \frac{\lambda_L}{2} (\| U \|_F^2 + \| V \|_F^2).
\end{equation}

\paragraph{Roadmap.}

We provide a survey of related work in Section~\ref{PWaC}. We consider two theoretical issues related to the smooth reformulation~\eqref{eq:smoothNC} in Section \ref{TheCon}. 
First, the factorization $L = UV^T$ gives a non-convex program that could introduce suboptimal stationary points and local minima. 
Second, it is important to understand the properties of the map $\varphi$ obtained by marginalizing over $S$ in the context 
of specific optimization algorithms. 
In Section \ref{Init}, we study heuristic initialization methods, which are important for non-convex programs. 
We present detailed numerical studies, comparing the non-convex method with competing alternatives, in Section~\ref{BurMontEx}. In this section, we also demonstrate the effectiveness of our non-convex formulation of low-rank recovery models for real-time video processing.
There are several approaches to nuclear norm minimization that focus on accelerating the SVD step, either by reducing its communication costs or by replacing it with a simpler subproblem, as in the Frank-Wolfe method. In Section \ref{sec:fastSVDcompare}, we develop two such methods for general regularized low-rank recovery models, and we compare these to our non-convex approach.
Our numerical experiments illustrate that our non-convex program is able to solve regularized low-rank recovery problems faster than existing solvers --- in many cases, by an order of magnitude --- without sacrificing accuracy in the solution.
We end with final thoughts in Section~\ref{Conc}.

\section{Prior Work and Contributions}
\label{PWaC}

Several authors have used alternating minimization to solve non-convex, factorized matrix completion models \citep{GAGG13,HardtAltMin,MatCompAltMin,KeshPhD,RegForMatComp,NoisyMatComp}. 
This line of research is supported by the existence of conditions that guarantee fast convergence to a globally optimal solution \citep{GAGG13,HardtAltMin,RegForMatComp,NoisyMatComp}, sometimes at a linear rate \citep{MatCompAltMin,GLRM}. 
These models are a special class of~\eqref{eq:smoothNC}, when no regularizer is present (i.e., $r(S) \equiv 0$). 

In \cite{GLRM}, the authors review the use of alternating minimization to solve regularized PCA problems of the following form:
\begin{equation}
\label{LowRankEq}
\min_{U,V} \quad \mathcal{L}(UV^T,X) + r(U) + \tilde{r}(V).
\end{equation}
Matrix completion models and RPCA can adopt the split-form in \eqref{LowRankEq}, but SPCP and other regularized models cannot. 
Our analysis applies to more general regularizers and offers an approach using first-order solvers rather than alternating minimization.

The authors of \cite{Shen2012} develop a split RPCA program and solve it using alternating minimization as well.
While their technique can be an order of magnitude faster than competitive convex solvers (such as \textcolor{black}{the inexact augmented Lagrangian method (IALM)} from \cite{IALM}), its performance suffers when the magnitude of the sparse component is large compared to the magnitude of the low-rank component.

In \cite{Procrustes}, the authors consider provable recovery of a low-rank solution to a system of linear equations using a non-convex, factorized program. In a similar vein, the work of \cite{CaramanisRPCA} extends these ideas to investigate a non-convex formulation of robust PCA that can be solved with gradient descent.

Another non-convex approach to RPCA is considered in \cite{NonConRPCA}. In contrast to the research discussed previously, 
their method does not use a low-rank factorization of $L$. Instead, these authors develop an algorithm of alternating projections, 
where each iteration sequentially projects $L$ onto a set of low-rank matrices and thresholds the entries of $S$. 
Their technique also comes with recovery and performance guarantees.

We develop a general non-convex formulation of regularized low-rank recovery models that can be solved using first-order optimization algorithms. 
We address problems associated with local minima and spurious stationary points that accompany non-convexity, 
by showing that all local minima of our model have the same objective value, 
and by providing a certificate to prove that a given solution is not a spurious stationary point.  
We also show that our method is particularly well-suited for the GPU.

\section{Theoretical Considerations}
\label{TheCon}

We discuss three issues related to the approach~\eqref{eq:smoothNC}.
First, we study the convexity and differentiability of the marginal function~\eqref{eq:genConvex} in Sections~\ref{subsec:conv} and \ref{Smooth}.
Next, moving from~\eqref{eq:genConvex} to~\eqref{eq:smoothNC} by factoring $L$, we transform a convex problem into a non-convex one. 
This creates two issues:  an implicit rank constraint $\rank(L) \le k$ that is not present in~\eqref{eq:smoothNC}, 
and potential for local minima and spurious stationary points in the non-convex variant that do not correspond to a solution of the convex problem.  
We address these issues in Section~\ref{sec:minima}. 

Section~\ref{sec:Certificate} introduces a computable certificate that can be used to check whether the non-convex model has converged to a global solution. Using the smoothness of $\varphi$, we derive a simple optimality certificate that can 
be tracked while optimizing~\eqref{eq:smoothNC} to identify either when the implicit rank of the factors is too small or  
when we have converged to a minimizer of~\eqref{eq:genConvex}.

\subsection{Convexity of $\varphi$}
\label{subsec:conv}

\noindent The convexity of the marginal function~\eqref{eq:genConvex} is a well-known result in convex analysis, see e.g.~\cite[Proposition 2.22]{VariationalAnalysis}. 
We include a short proof of this fact.  

\begin{lemma}
Let $f : \mathbb{R}^{m \times n} \times \mathbb{R}^{m \times n} \to \mathbb{R}$ be a convex function, and define $g(L) = \min_{S \in \mathbb{R}^{m \times n}} f(L,S)$ for $L \in \mathbb{R}^{m \times n}$. The function $g$ is convex.
\label{convexity}
\end{lemma}
\begin{proof}
For all $L_1, L_2, S_1, S_2 \in \mathbb{R}^{m \times n}$ and $t \in (0,1)$,
\begin{align*}
g(t L_1 + (1-t) L_2) &= \min_{S \in \mathbb{R}^{m \times n}} f(t L_1 + (1-t) L_2, S) \\
&\le f(t L_1 + (1-t) L_2, t S_1 + (1-t) S_2) \\
&\le t f(L_1,S_1) + (1-t) f(L_2,S_2).
\end{align*}
As this inequality holds for arbitrary $S_1$ and $S_2$, it holds for any minimizing $S^{\star}$:
\begin{align*}
g(t L_1 + (1-t) L_2) &\le t f(L_1,S^{\star}) + (1-t) f(L_2,S^{\star}) \\
&= t g(L_1) + (1-t) g(L_2).
\end{align*}
\end{proof}
\noindent It follows immediately that $\varphi$ as defined in \eqref{PhiDef} is convex when $\mathcal{L}$ and $r$ are convex. 

\subsection{Smoothness of $\varphi$ }
\label{Smooth}

When $\mathcal{L}$ is the least-squares error and $\mathcal{A}$ is an injective linear operator, Lipschitz differentiability of $\varphi$ is immediate from the strong convexity of $\mathcal{L}$. Support for this claim is given, for example, in \cite{CombettesBook}, where the relevant proposition for the specific case where $\mathcal{A}$ is a reshaping operator is as follows:
\begin{proposition*}
\textbf{\cite[Prop. 12.29]{CombettesBook}:} Let $r$ be a proper lower-semicontinuous convex function, and let $ ^{\mu}r (\mathcal{A}(L) -\mathrm{vec}(S)- b) = \inf_{S} \ r(S) + \frac{1}{2 \mu} \| \mathcal{A}(L) -\mathrm{vec}(S)- b \|_2^2$ be the Moreau envelope of $f$ with parameter $\mu$. 
Then $ ^{\mu}r$ is differentiable and its gradient is $\mu^{-1}$-Lipschitz continuous.
\end{proposition*}

More generally, $\varphi(L)$ is differentiable as long as $r(S)$ has a unique minimizer and the loss function $\mathcal{L}$ is convex and smooth. 
By~\cite[Theorem 10.58]{VariationalAnalysis}, we have 
\begin{equation}
\partial \varphi(\cdot) = \mathcal{A}^*\nabla\mathcal{L}\left(\mathcal{A}(\cdot) -b - \mathcal{A}(S^{\star})\right),
\label{eq:diff}
\end{equation}
where $S^{\star}$ is any minimizer of \eqref{eq:lag}. The subdifferential is a singleton if and only if the minimizing $S^{\star}$ is unique, 
and a unique minimum is guaranteed if the sum $\mathcal{L}(\mathcal{A}(L + S) - b) + r(S)$ in \eqref{eq:lag} is strictly convex.

Using SPCP as an example, we have
\[
\varphi(L) = \min_{S} \left\{\frac{1}{2} \|\mathcal{A}(L)+\mathrm{vec}(S)-b\|_2^2 + \lambda_S \|S\|_1\right\},
\]
where sparse outliers $S$ are fully observed. 
The objective is strongly convex in $S$, and has the unique minimizer
\[
S^{\star} = \mathrm{prox}_{\lambda_S \|\cdot\|_1}(\mathcal{A}(L) - b).
\]
Therefore, $\varphi(L)$ is smooth, with 
\[
\nabla \varphi(L) = \mathcal{A}^*\left(\mathcal{A}(L)+\mathrm{vec}(S^{\star})-b\right).
\]
We used~\eqref{eq:diff} rather than the parametric form of $\varphi(L)$ to obtain its value and gradient. Using \eqref{eq:factChar} to replace the nuclear norm and $\varphi$ to eliminate the non-smooth regularizer, we now have a smooth optimization problem amenable to first-order methods.

For minimizing smooth objectives, quasi-Newton methods are often faster than algorithms that require only first-order smoothness, such as gradient-descent. The convergence rates of quasi-Newton methods depend on second-order smoothness, which does not hold in problems such as SPCP, since e.g. the Huber function is only $\mathcal{C}^1$. However, empirical results (including those presented in Section \ref{BurMontEx}) suggest that these methods are effective.

\subsection{Rank and Local Minima}
\label{sec:minima}

The factorized problem~\eqref{eq:smoothNC} is not equivalent to~\eqref{eq:genConvex} because of the implicit rank constraint. 
However, we can show that~\eqref{eq:smoothNC} is equivalent to the rank-constrained problem
\begin{equation}
\min_{L} \quad \lambda_L \| L \|_* + \varphi(L) \quad \textnormal{subject to:} \quad \rank(L) \le k.
\end{equation}
This follows immediately from~\eqref{eq:factChar}. In order for the split problem to recover the solution to \eqref{eq:comb}, it must be initialized with a $k$ larger than the rank of the minimizing $L$. However, a larger $k$ slows computation, so it must be chosen with care. This issue is considered in more depth in Section \ref{Init}.

We want to be sure that any local minimum of~\eqref{eq:smoothNC} corresponds to a (global) minimum of~\eqref{eq:genConvex}. 
The following theorem combines ideas from \cite{BurerMonteiro2005} and \cite{Aravkin2014} to show this holds 
provided that $k$ is larger than the rank of the minimizer for~\eqref{eq:genConvex}.
\begin{theorem}
\label{NoLoca}
Consider an optimization problem of the following form:
\begin{equation} \label{eq:1A1}
\min_{\X \succeq 0} \quad f(\X), \quad \textrm{such that}\quad \rank(\X)\le \rankVar,
\end{equation}
where  $\X \in \mathbb{R}^{\dimVarX}$ is a positive semidefinite real matrix, and $f$ is a lower semi-continuous function mapping to $[-\infty,\infty]$ and has a non-empty domain over the set of positive semi-definite matrices. 
Using the change of variable $\X = \P\P^T$, take $\P \in \mathbb{R}^{\dimVar \times \rankVar}$, and consider the problem   
\begin{equation} \label{eq:1A2}
\min_{\P} \quad g(\P) \defeq f(\P\P^T).
\end{equation}
Let $\bar \X = \bar \P \bar \P^T$, where $\bar \X$ is feasible for~\eqref{eq:1A1}. Then $\bar \X$ is a local minimizer of~\eqref{eq:1A1} if and only if $\bar \P$ is a local minimizer of~\eqref{eq:1A2}.
\end{theorem}
The proof of Theorem \ref{NoLoca} is deferred to Appendix A. Using the SDP formulation of the nuclear norm presented in \cite{fazel2001rank}, 
our problem can be recast as a semi-definite program so that we can apply Theorem~\ref{NoLoca}. Define 
\begin{equation}
Z = \left[ \begin{array}{c}
U \\
V \\
\end{array} \right]
\left[ \begin{array}{c}
U \\
V \\
\end{array} \right]^T
= \left[ \begin{array}{cc}
UU^T & L    \\
L^T    & VV^T \\
\end{array} \right].
\end{equation}
The matrix $Z$ is positive semi-definite, and has form $Z = \P \P^T$, with
\(
P = \left[ \begin{array}{c}
U \\
V \\
\end{array} \right].
\)
Let $\mathcal{R}(\cdot)$ be the function that extracts the upper-right block of a matrix (so that $\mathcal{R}(Z) = L$), and let 
\begin{equation}
f(Z) = \frac{\lambda_L}{2} \trace(Z) + \varphi( \mathcal{R}(Z) ).
\end{equation}
Using equation (3), we now see that the rank-constrained problem is equivalent to
\begin{equation}
\min_{Z \succeq 0} \quad f(Z), \quad \textnormal{such that} \quad \rank(Z) \le k.
\end{equation}
Applying Theorem \ref{NoLoca}, we can now be assured that $\overline{P}$ is a local minimizer to the split program if and only if $\overline{Z}$ is local minimizer of the original problem. This is equivalent to the statement that the point $(\overline{U},\overline{V})$ is a local minimizer of the split problem if and only if $\overline{L} = \overline{U} \ \overline{V}^T$ is a local minimizer of the original, rank-constrained program. 

\subsection{Certificate of Convergence}
\label{sec:Certificate}

Although Theorem \ref{NoLoca} asserts that the split problem and the rank-constrained problem have the same local minima, the rank-constrained problem is itself non-convex, so we have not guaranteed that every stationary point of the split problem solves the convex program.
Using an approach that builds on \cite{GLRM}, we develop a certificate to check whether a recovered solution to (6) corresponds to a spurious stationary point when $\varphi$ is convex (in particular when $\mathcal{L}$ and $r$ are convex). This technique can also be used on the convex formulation as a method to check the distance to optimality.

We base our certificate on the following analysis. Let $F$ be any proper convex function, then Fermat's rule states that the set of minimizers of $F$ are points $L$ such that $0 \in \partial F(L)$,
where $\partial(\cdot)$ denotes the subdifferential. One could take a candidate point $L$, construct $\partial F(L)$, and if it contains $0$, conclude that $L$ is the minimizer; this is not realistic for large-scale problems. 
Instead we find
\newcommand{\Eps}{\mathcal{E}}
\begin{equation}
\Eps \in \partial F(L)
\end{equation}
where $\Eps$ is small. Let $L^\star$ denote any minimizer of $F$, then from the definition of a subgradient, 
$F(L^\star) \ge F(L) + \< \Eps, L^\star - L \>$.
Then 
\begin{equation} \label{eq:Fbound}
F(L) - F(L^\star) \le \< \Eps, L-L^\star \> \le \|\Eps\|_p \cdot \|L-L^\star\|_d
\end{equation}
for any pair of primal and dual norms $\|\cdot\|_p$ and $\|\cdot\|_d$. We discuss bounding $\|\Eps\|_p$ and $\|L-L^\star\|_d$ in the next two subsections.

\paragraph{Finding an approximate zero $\Eps$}
The key to the certificate is analyzing the marginalized problem \eqref{eq:genConvex} in terms of $L$, rather than the problem \eqref{eq:gen} in terms of $(L,S)$. 
The point $L=UV^T$ is the optimal solution of~\eqref{eq:genConvex} if and only if
\begin{equation}
\label{eq:reducedSub}
0 \in \partial \big( \| L \|_* + \varphi( L ) \big).
\end{equation}
We have set $\lambda_L=1$ for convenience, or alternatively one can absorb $\lambda_L^{-1}$ into $\varphi$. 
Since both the nuclear norm and $\varphi$ are proper lower semi-continuous and the intersection of the interior of their domains is non-empty (in particular they are finite valued), 
then by \cite[Cor.\ 16.38]{CombettesBook}, we have that
\begin{equation*}
0 \in \partial \big( \| L \|_* + \varphi( L ) \big) \quad \Leftrightarrow \quad 0 \in \big(  \partial \| L \|_* \big) + \big( \partial \varphi( L ) \big).
\end{equation*}
Both of these subdifferentials are computable.  Let 
\newcommand{\bigU}{\widetilde{U}}
\newcommand{\bigV}{\widetilde{V}}
\newcommand{\Uperp}{U_2}
\newcommand{\Vperp}{V_2}
\newcommand{\zero}{0}
\newcommand{\eye}{I}
\begin{equation}
L = \bigU \Sigma \bigV^T = \left[ \begin{array}{c c}
U_1 & \Uperp
\end{array} \right]
\left[ \begin{array}{c c}
\Sigma_1 & \zero \\
\zero & \zero
\end{array} \right]
\left[ \begin{array}{c}
V_1^T \\
\Vperp^T
\end{array} \right]
\end{equation}
be the (full) SVD of $L$, where $U_1\in\R^{m\times r}$ and $V_1\in\R^{n\times r}$.
The subdifferentials of the nuclear norm at $L$ comprises matrices of the form $X=U_1 V_1^T + W$, where $U_1,V_1$ contain the left and right singular vectors of $L$ that correspond to non-zero singular values,%
\footnote{The orthogonal matrices $U_1$ and $V_1$ should not be confused with the variables $U$ and $V$ that are used as a factorization of $L$, as $U$ and $V$ are not necessarily orthogonal. In fact, $U_1$ and $V_1$ can be efficiently computed from the factorization $L = UV^T$ by taking the QR-decompositions $U = Q_U R_U$ and $V = Q_V R_V$ so $L = Q_U (R_U R_V^T) Q_V^T$. Perform a SVD on the small inner matrix to write $(R_U R_V^T) = U_R\Sigma V_R^T$, and hence $\Sigma$ are the nonzero singular values of $L$ and $U_1 = Q_UU_R$ and $V_1 = Q_V V_R$ are the corresponding left and right singular vectors.
}
and $W$ satisfies the conditions $U_1^T W = 0$, $W V_1 = 0$ and $\| W \|_2 \le 1$.
Equivalently, 
\begin{equation} \label{eq:subdiff}
X\in\partial\|L\|_* \quad\text{if and only if}\quad
\exists W'\,\in\R^{m-r \times n-r} \,\text{s.t.} \;
\bigU^T X \bigV = 
\left[ \begin{array}{c c}
\eye & \zero \\
\zero & W'
\end{array} \right], \quad \|W'\| \le 1.
\end{equation}
We can find an explicit matrix $D\in \partial\varphi(L)$ using \eqref{eq:diff}, and it is unique under the smoothness conditions described in Section~\ref{Smooth}.

Altogether, we can guarantee a point $\Eps \in \partial \|L\|_* + \partial\varphi(L)$ such that
\begin{align}
\|\Eps\|_F^2 &= \min_{ X \in \partial \|L\|_*} \, \| X + D \|_F^2 \notag \\ \notag
&= \min_{ X \in \partial \|L\|_*} \, \|\bigU^T \left( X + D\right) \bigV \|_F^2 \\ \notag
&= \min_{ \|W'\|\le 1 } \, \left\| \left[ \begin{array}{c c}
\eye & \zero \\ \notag
\zero & W'
\end{array} \right] +  \bigU^T D \bigV \right\|_F^2 \\
&=  
    \| \eye - U_1^TDV_1 \|_F^2 + \|U_1^TD\Vperp\|_F^2 +
    \|\Uperp^T DV_1 \|_F^2 
    +  \min_{\|W'\|\le 1 } \, \| \Uperp^T D \Vperp - W' \|_F^2 .
\end{align}
In the first equation, the minimum is achieved since the squared-norm is continuous and the set is compact, and \eqref{eq:subdiff} is used in the third equation. 
Each term of the fourth equation can be calculated explicitly, with the last term obtained via projection onto the unit spectral-norm ball, which requires the SVD of the small matrix $\Uperp^T D \Vperp$. 
This bound on $\|\Eps\|_F$ can then be used in \eqref{eq:Fbound}.

Most terms above can be computed efficiently, in the sense that $\Uperp$ and $\Vperp$ never need to be explicitly computed (which is most important when $r$ is small, since then $\Uperp$ and $\Vperp$ have $m-r$ and $n-r$ columns, respectively), and therefore the computation consists only of matrix products, thin-QR factorizations, and an $r\times r$ SVD, leading to a complexity of $\mathcal{O}( rmn + r^2\cdot(m+n) + r^3)$.
For example, we compute
\[
\| U_1^T D \Vperp \|_F^2 = \|U_1^T D\|_F^2 - \|U_1^T D V_1\|_F^2
\]
and $\|\Uperp^T D V_1\|_F^2$ is computed analogously.
The final term requires an unavoidable $(m-r)\times (n-r)$ SVD factorization, so for large $m,n$ the certificate may be computed as a final check rather than at each iterate.

\paragraph{Bounding the distance to the feasible set}
We seek a bound on $\|L-L^\star\|_d$ in an appropriate norm, which together with \eqref{eq:Fbound} gives us a bound on the objective function difference $F(L) - F(L^\star)$. Since the bound on $\Eps$ is in the Frobenius norm, ideally we set $\|\cdot\|_d = \|\cdot\|_F$, but bounds in any norm will work using $\|L\|_F \le \sqrt{\text{rank}(L)}\|L\|$ and $\|L\|_F \le \|L\|_*$.

Letting $F(L) = \|L\|_* + \varphi(L)$, we first bound $F(L^\star)$ by computing $F(L)$ for explicit choices of $L$. To be concrete, in this section we assume $r(S) = \|S\|_1$ with $\lambda_S=1$, and $\mathcal{L}(\cdot) = \frac{1}{2}\|\cdot\|^2$. Choosing $L$ such that $\mathcal{A}(L)=b$ means that $S=0$ in the definition of $\varphi$ \eqref{PhiDef}, and hence $F(L) = \|L\|_*$. Choosing $L=0$ and $S=0$ gives the bound $F(L) \le \frac{1}{2}\|b\|^2$. A third choice is to explicitly compute $F(L)$ at all iterates in the algorithm and record the best.

\newcommand{\Fbnd}{F_\text{bound}}
Denoting $F^\star = \min_L F(L)$ and using $F^\star \le \Fbnd$, 
non-negativity of $\varphi$ immediately implies $\|L^\star\|_* \le \Fbnd$.
Hence $\|L-L^\star\|_F \le \|L\|_F + \|L^\star\|_F \le \|L\|_F + \|L^\star\|_* \le \|L\|_F + \Fbnd$, and $\|L\|_F$ is explicitly computable.

\paragraph{Results}
Figure \ref{CertTests} shows how the distance to the optimal subgradient decreases over time for the convex solver LagQN, the non-convex Split-SPCP with a rank bound that is too strict, and Split-SPCP with a rank bound large enough to reach the global optimum. When the rank bound is too restrictive, Split-SPCP cannot recover the optimal $L$. Measuring the distance to the optimal $L$ reveals this. Figure \ref{CertTests} shows that the distance plateaus far from zero when the rank of $L$ is bounded above by 30, and the rank of the optimizing $L$ is 58. In practice, this measure can be used to indicate that the rank bound should be increased.

\begin{figure}
\includegraphics[width=0.4\linewidth]{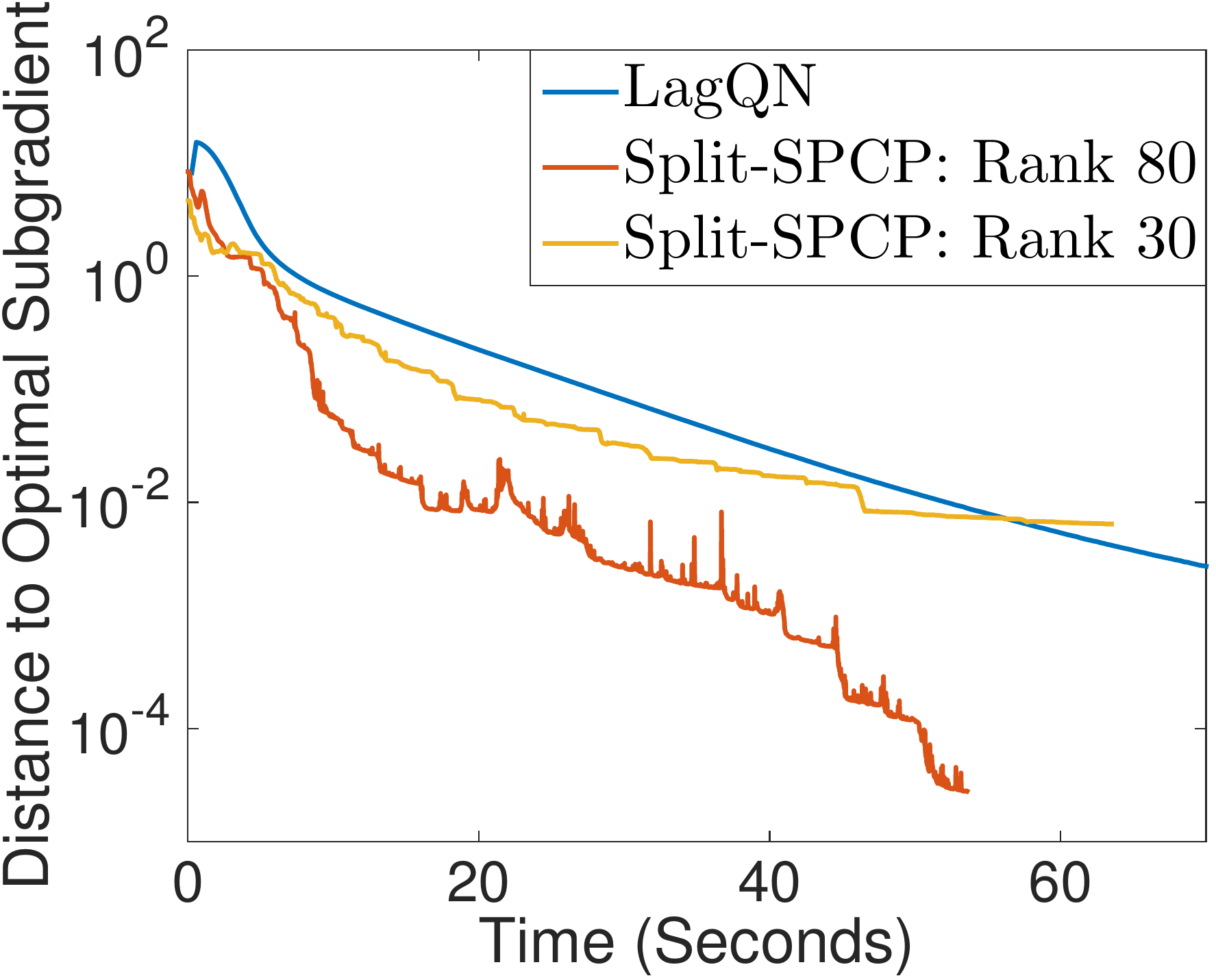}
~
\includegraphics[width=0.4\linewidth]{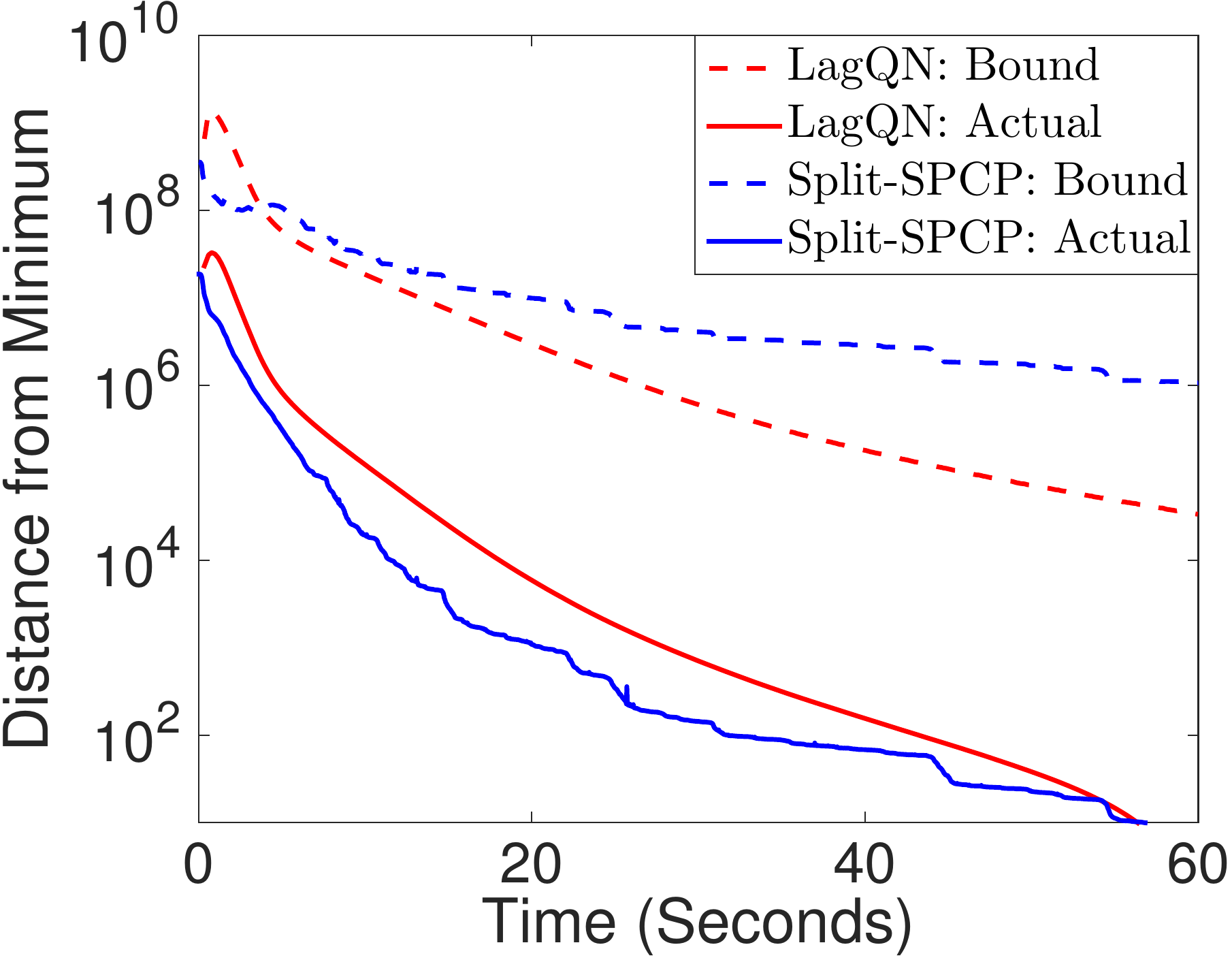}
~~~~
\caption{(Left): Distance from the optimal subgradient for the convex solver LagQN and the non-convex Split-SPCP with different rank bounds. The rank of the optimal $L$ is 58. (Right): Bounded distance from minimal objective value, as given by $\| \Eps \|_F (\|L\|_F + F_{\textnormal{bound}})$, with $F_{\textnormal{bound}}$ given by the objective value at the current iterate. The actual distance from the minimal objective value is shown for comparison. The rank bound for the Split-SPCP test is 80.} 
\label{CertTests}
\end{figure}

\section{Initialization}
\label{Init}

The non-convex formulation~\eqref{eq:smoothNC} is sensitive to the initial point, and requires a bound on the rank of $L$. In this section, we present heuristics for choosing a good initial point and determining an appropriate rank-bound.

\subsection{Dynamically Increasing $k$}

Figure \ref{Rank} demonstrates the sensitivity of the Split-SPCP program to the factor rank $k$. These tests were performed on the surveillance video data described in Section \ref{BurMontEx}, where the true rank of the low-rank component is 58. We see that when $k \le \rank(L)$, Split-SPCP does not converge to the correct solution, but if $k$ is much larger than the rank of the minimizing $L$, then the computation is slowed significantly. 

Because our solver is oblivious to the individual dimensions of $U$ and $V$, columns can be added to both matrices on the fly. Dynamically updating the rank bound can help when~\eqref{eq:smoothNC} is initialized with a $k$ that is too small. 
Suppose a solver reaches some convergence criterion at iteration $i$. To see if this is a solution to~\eqref{eq:genConvex}, we add a single column to $U$ and $V$: 
\begin{equation}
\left[ \begin{array}{cc}
U_i & u
\end{array} \right]
\left[ \begin{array}{cc}
V_i & v
\end{array} \right]^T
= L_i + uv^T,
\end{equation}
and observe whether this rank-one allows for a lower objective value or certificate value.
Dynamically increasing $k$, even aggressively, is more efficient than overestimating $k$ from the start. 

\begin{figure}[h!]
\centering
\includegraphics[width=0.5\linewidth]{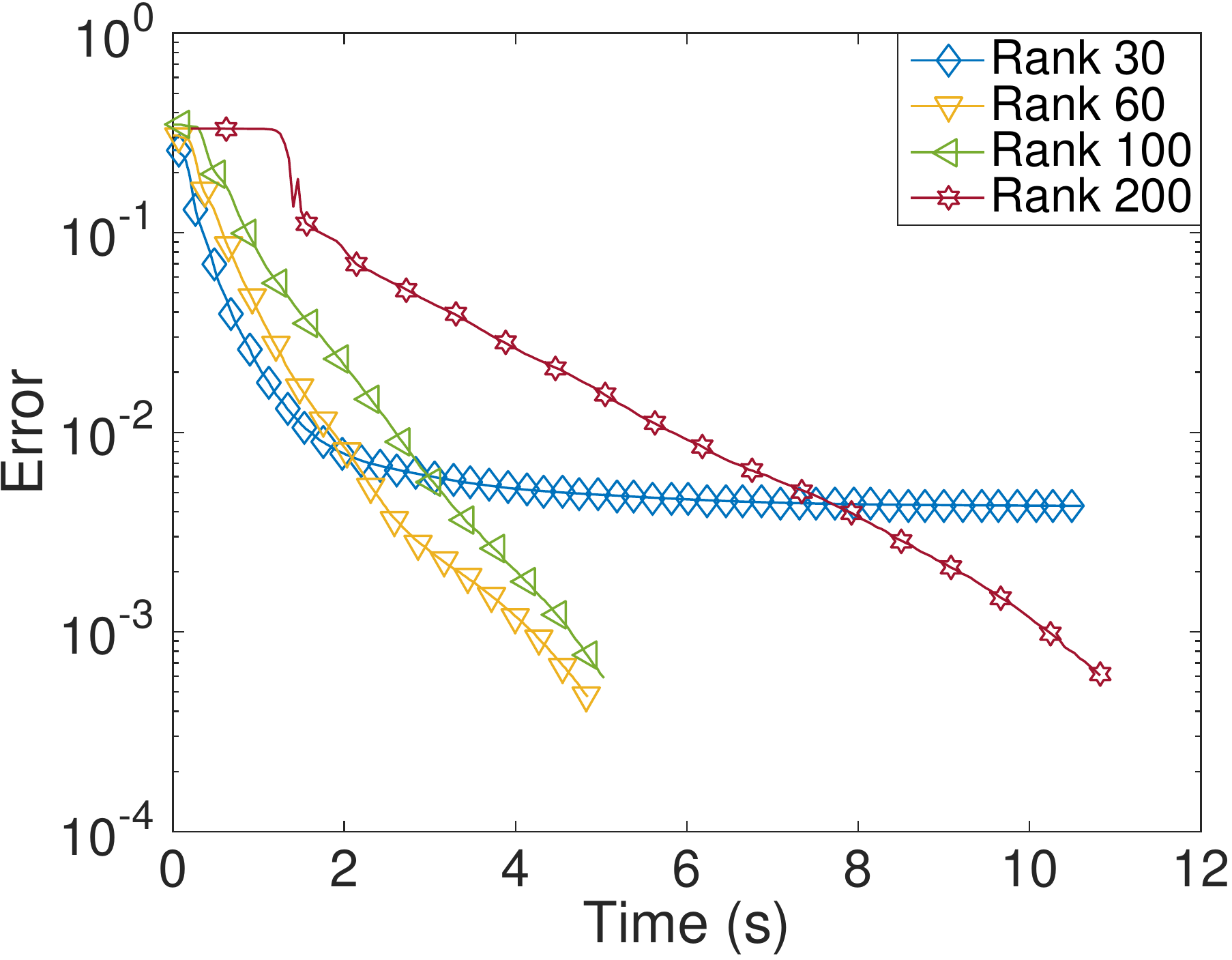}
\caption{Varying the rank bound $k$ in Split-SPCP affects computation time and the accuracy of the final solution. One marker represents 10 iterations.}
\label{Rank}
\end{figure}

\subsection{Choosing $U_0$ and $V_0$}

Choosing an initialization point $(U_0,V_0)$ has received considerable attention in factorized low-rank recovery problems. 
Most of this work has centered around matrix completion using alternating minimization, see, for example, \cite{HardtAltMin,MatCompAltMin,RegForMatComp,NoisyMatComp,GLRM}.
Provably good choices for $U_0$ and $V_0$ are $U_0 = U \Sigma^{\frac{1}{2}}$ and $V_0 = V \Sigma^{\frac{1}{2}}$, where $U$ and $V$ come from the SVD of $\mathcal{A}^* (X)$, where $\mathcal{A}$ is a sampling operator and $\mathcal{A}^*$ is its adjoint. 
\cite{HardtAltMin} showed that these initial points lie in the basin of attraction of globally optimal minima, yielding fast convergence to an optimal solution. 

As mentioned in \cite{GLRM}, it is sometimes not necessary to perform a full SVD to form $U_0,V_0$. Once $k$ is chosen, only a partial SVD is necessary to calculate the first $k$ singular values and vectors, which can be done efficiently using the randomized SVD (rSVD)~\citep{structureRandomness}. Although using the rSVD is significantly faster than a full SVD, the values and vectors it returns are not always accurate, especially when singular values of the underlying data do not decay rapidly.

\begin{figure}[h!]
\centering
\includegraphics[width=0.5\linewidth]{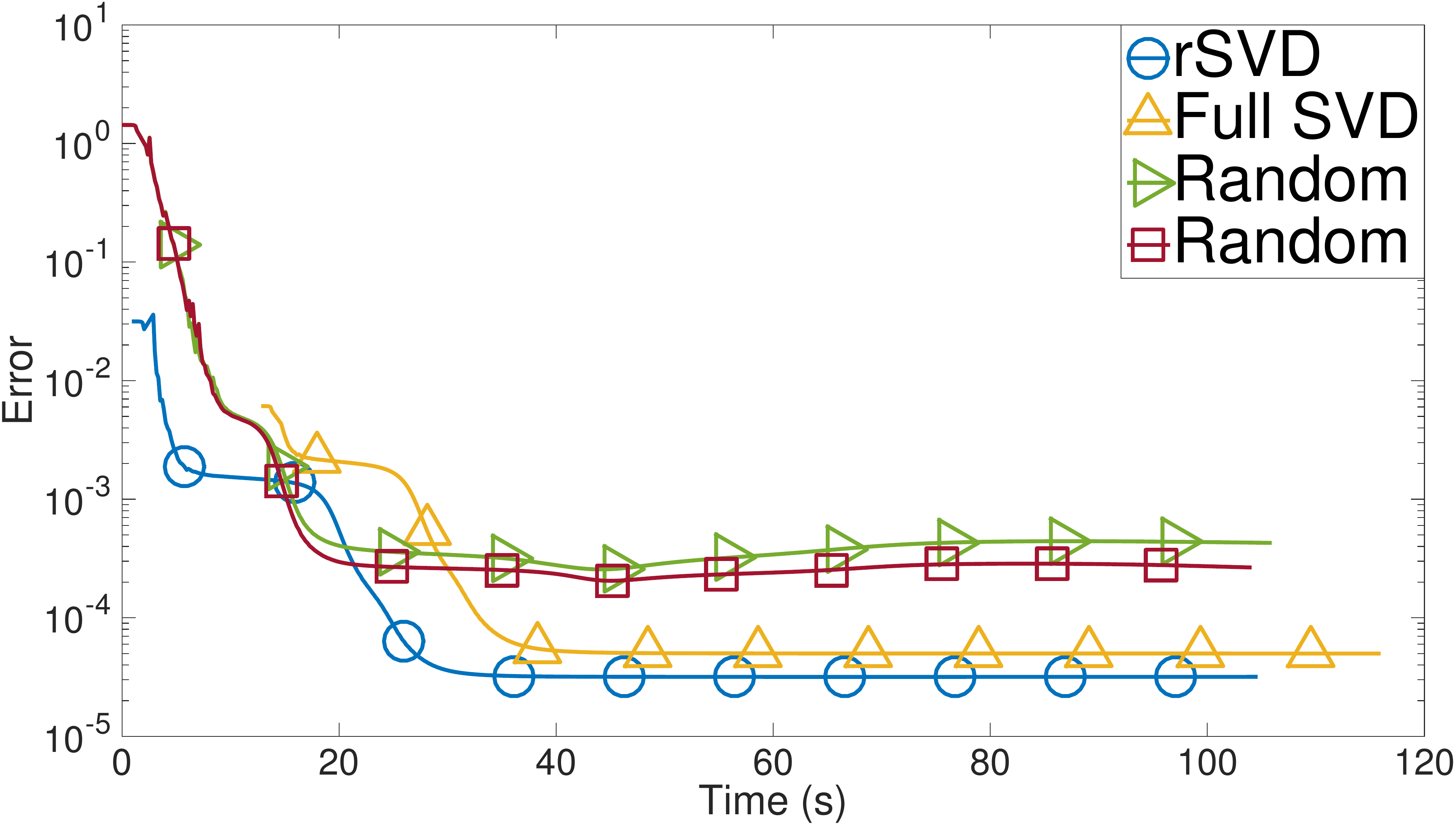}
\caption{The performance of Split-SPCP with various initial points. All tests were run on a 5,000 $\times$ 5,000 matrix $X = L + S$, where $\rank(L) = 1000$ and $\textnormal{sparsity}(S)$ = 9.43\%. The rank bound was $k = 1,\!050$. One marker represents 50 iterations.}
\label{Fig:Init}
\end{figure}

Figure \ref{Fig:Init} shows the performance of Split-SPCP using various initial points.
\textcolor{black}{The tests  were performed on a $5,\!000 \times 5,\!000$ matrix $X = L + S$, where $\rank(L) = 1000$ and $\textnormal{sparsity}(S)$ = 9.43\%. The rank bound was $k = 1,\!050$. All $5,\!000$ singular value and vector triples of $X$ were calculated for the SVD test, and $U_0$ and $V_0$ were formed from the first $k = 1,\!050$ triples. For the rSVD test, only the first $1,\!050$ triples were approximated, $U_0$ and $V_0$ were formed from these.}

Figure \ref{Fig:Init} shows that initializing $U_0$ and $V_0$ with the first 1,050 singular value and vector triples returned by the full SVD yields the smallest initial error, but the extra computational cost is not worth the better start. The rSVD allows quicker convergence to a similar solution. The two random initializations converge to a stationary point that is not globally optimal.

\section{Numerical Experiments}
\label{BurMontEx}

In this section, we present numerical experiments that illustrate the speed of Split-SPCP. Recall that the Split-SPCP program is given by
\begin{equation}
\label{Thm5}
\begin{split}
\min_{U,V}& \quad \frac{\lambda_L}{2} (\|U\|^2_F + \|V\|_F^2) + \varphi(UV^T),
\end{split}
\end{equation}
where
\begin{equation}
\varphi(UV^T) = \min_S \frac{1}{2} \| UV^T + S - X \|_F^2 + \lambda_S \| S \|_1.
\end{equation}
Since the objective of Split-SPCP is smooth, any first-order method can be used to solve it. Also, because the communication-heavy SVD step is no longer a limitation, we are motivated to choose a solver and an implementation that are most suited for the GPU.

\subsection{Implementation Details}
\label{subsec:ImpDets}

To find a solver that fits our Split-SPCP program well, we compare the performance of several first-order methods as they solve a logistic regression problem on the GPU. Each of the solvers we use in this test are modified from \cite{SchmidtCode}, using MATLAB's Parallel Computing Toolbox, which in turn uses MAGMA libraries \citep{Magma} to run on the GPU. For all GPU computation, we use a Tesla K40c GPU card with thread block sizes $[1024, 1024, 64]$, grid sizes $[2.15 \times 10^9, 6.6 \times 10^4, 6.6 \times 10^4]$, a clock rate of 745 MHz, and 12.0 GB of total memory. For all of the tests run on the CPU in this section, we compute on an Intel Xeon CPU E5-2630 version 3, using a 64-bit architecture and a clock rate of 2.40GHz. Our version of MATLAB is R2016b.

The time it takes for various solvers to minimize a logistic loss with a Tikhonov regularization term for $10^4$ labels is shown in Table \ref{Table:LogLoss} below.
We measure performance in ``computation time'', the time it takes for each solver to come within $10^{-8}$ of the optimal objective value.

\begin{figure}[h!]
\centering
\begin{tabular}{c c}
\toprule
Solver &  Normalized Computation Time \\ 
\midrule
Cyclic Steepest Descent & 3.46 \\ 
Barzilai-Borwein & 1.24 \\ 
Conjugate Gradient & 1.44 \\ 
Scaled Conjugate Gradient & 1.53 \\ 
Preconditioned Conjugate Gradient & 3.20 \\ 
{\bf L-BFGS (10 iterations in memory)} & {\bf 1} \\ 
L-BFGS (50 iterations in memory) & 1.82 \\ 
BFGS & 2.39 \\ 
Hessian-Free Newton's Method & 3.05 \\ 
\bottomrule
\end{tabular}
\captionof{table}{First-order methods solving logistic loss program with $10^4$ labels on the GPU. All times are reported as a ratio with respect to the fastest solver (L-BFGS with 10 iterations in memory). Time is recorded when the solution is within $10^{-8}$ of the optimal value.}
\label{Table:LogLoss}
\end{figure}

\vspace{7mm}

\par

\indent We see from our analysis that the L-BFGS solver, storing a small number of iterations in memory, yields the best performance. Of course, the relative performance of different solvers on the GPU depends on the stopping tolerance as well as the specific problem, but there is further evidence in the literature supporting L-BFGS as a competitively efficient solver \citep{WrightBook}.

\subsection{SPCP for Background Subtraction}
\label{subsec:BackSubtract}

In \cite{ThierryVideoReview} and \cite{Sobral}, the authors provide a review of the fastest algorithms in the family of robust PCA and principal component pursuit for background subtraction. We choose some of the fastest of these to compare to our Split-SPCP algorithm. Each of these methods runs faster on the GPU than the CPU, so we run all of them on the GPU for our comparisons (using MATLAB's Parallel Computing Toolbox). We find that our non-convex solver is particularly well-suited for the GPU, and {\bf outperforms all other solvers in almost every case.}

For our background subtraction tests, we use the escalator \textcolor{black}{surveillance video} provided by \cite{LiData}. 
We want to identify the people in the video while ignoring the moving escalators \textcolor{black}{and the stationary background}. 
This problem is particularly difficult for many low-rank recovery algorithms, such as PCA, because the motion of the escalator is a confounder. 
SPCP is less sensitive to outliers, so it can overcome this challenge.

\begin{figure}
\begin{floatrow}
\ffigbox{%
\includegraphics[width=0.85\linewidth]{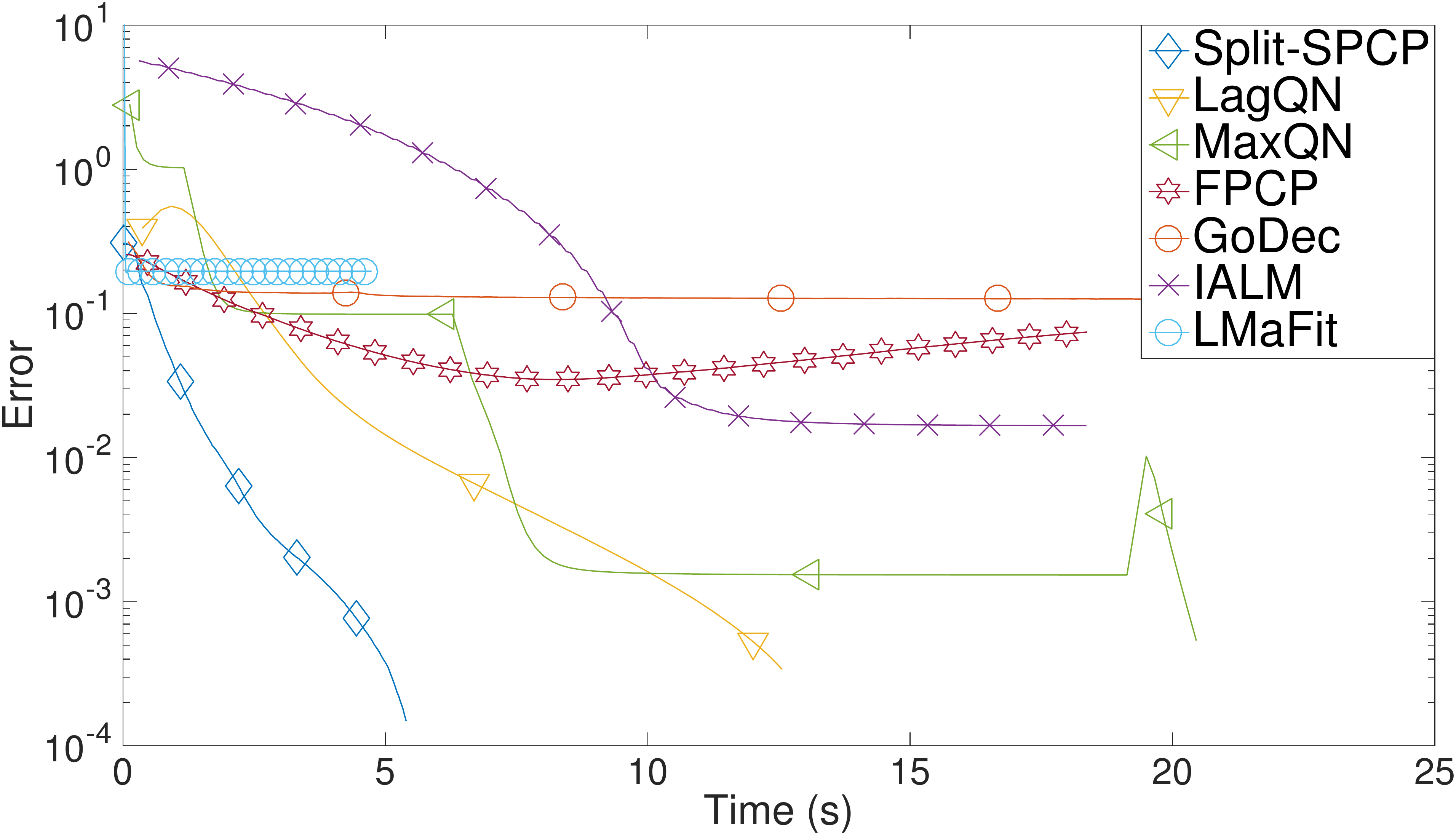}}
{%
\caption{Comparing SPCP solvers for background subtraction. One marker corresponds to 40 iterations.}
\label{BackSub}}
\capbtabbox{%
\resizebox{0.48\textwidth}{!}{%
\begin{tabular}{C{5.7cm} C{4.7cm}}
Algorithm & Reference \\ \hline \\
{\bf Split-SPCP} & {\bf This Work} \\ \vspace*{1mm}
Lagrangian Quasi-Newton (LagQN) & \cite{aravkin2014variational} \\ \vspace*{1mm}
Max Quasi-Newton (MaxQN) & \cite{aravkin2014variational} \\ \vspace*{1mm}
Fast PCP (FPCP) & \cite{FPCP} \\ \vspace*{1mm}
Go Decomposition (GoDec) & \cite{GoDec} \\ \vspace*{1mm}
Inexact ALM (IALM) & \cite{IALM} \\ \vspace*{1mm}
LMaFit (LaMaFit) & \cite{LMaFit} \\ 
\vspace*{3mm}
\end{tabular}}
}{%
  \caption{References for the algorithms used in testing. Among these are the fastest models for background subtraction, as determined in \cite{ThierryVideoReview} and \cite{Sobral}.}%
}
\end{floatrow}
\end{figure}

To find a reference solution, we hand-tune the parameters $\lambda_L$ and $\lambda_S$ in the (convex) quasi-Newton Lagrangian SPCP algorithm (LagQN) until we find a qualitatively accurate decomposition. This process produces a reference low-rank component $L_{ref}$ and sparse component $S_{ref}$. The $L_{ref}$ we choose has rank 58, and $S_{ref}$ is 58.03\% sparse. 
The optimal parameters for both Split-SPCP and LagQN are $\lambda_L = 115$ and $\lambda_S = 0.825$. 
We tune the parameters in the other solvers to recover this solution as closely as possible. 
We measure error as the normalized difference between the objective and the objective at the reference solution. 
Since the solvers minimize different objectives, we calculate the Split-SPCP objective value at each iteration for every algorithm, and we do not include these calculations in our time measurements. 
To initialize Split-SPCP, we use the first 100 singular values and vectors from the randomized SVD of the data matrix.

Several of the algorithms in Figure \ref{BackSub} do not converge to the same solution, despite considerable effort in parameter tuning. 
These algorithms were designed to quickly find approximate solutions to the SPCP problem, and might suffer from the large amount of noise present in the data.  
The approximate solutions recovered by these algorithms are qualitatively different from the Split-SPCP solution, see Figure~\ref{Pix}. 
We also find that a lower objective value generally corresponds to a qualitatively superior solution.  
\begin{figure}[h!]
\centering
\begin{subfigure}[b]{0.49\linewidth}\includegraphics[width=0.9\linewidth]{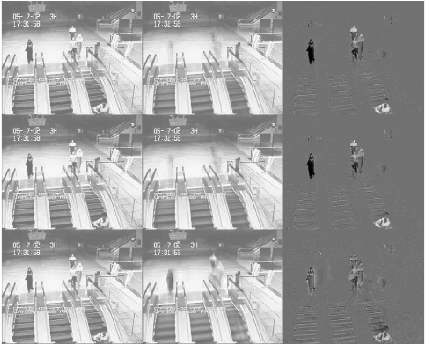}
\caption{X,L, and S matrices found by (from top to bottom) Split-SPCP, LagQN, and GoDec.}
\end{subfigure}
~
\begin{subfigure}[b]{0.49\textwidth}\includegraphics[width=0.9\linewidth]{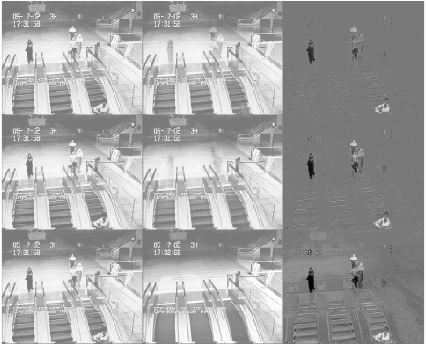}
\caption{X,L, and S matrices found by (from top to bottom) FPCP, IALM, and LMaFit.}
\end{subfigure}
\caption{Background subtraction using various SPCP solvers on surveillance video data from \cite{LiData} (frame 10 is shown). We see that Split-SPCP, LagQN, and IALM best locate the people while ignoring the escalators.}
\label{Pix}
\end{figure}

To quantitatively measure the quality of the solutions, we use the {\it Corrected Akaike Information Criterion}, or AIC$_{\textnormal{\large c}}$. The AIC$_{\textnormal{\large c}}$ measures the fit of a statistical model to a certain data set. Given a model with $p$ parameters and a log-likelihood function $\ell$, the value of the AIC$_{\textnormal{\large c}}$ is
\[\textnormal{AIC}_{\textnormal{\large c}} =  2 (p - \log(\ell_{\max})) + \frac{2p(p+1)}{m \cdot n-p-1}, \]
where $m \cdot n$ is the size of the data set and $\ell_{\max}$ is the maximum of $\ell$. The preferred statistical model is the one that minimizes the AIC$_{\textnormal{\large c}}$. The AIC$_{\textnormal{\large c}}$ favors models that maximize the likelihood function and penalizes complex models with many parameters, guarding against overfitting. Using AIC$_{\textnormal{\large c}}$ as a measure of quality
avoids both overemphasis of objective values and the inherent ambiguity of visual comparisons.

To compute the AIC$_{\textnormal{\large c}}$ value, we must formulate SPCP as a statistical model. It is well-known that the least-squares loss term assumes that the data is of the form $X = L+S+Z$, where the entries of $Z$ are i.i.d. Gaussian random variables with $\mu = 0$ and variance estimated using the sample variance, $\hat{\sigma}^2 = \frac{\| X \|^2_F}{m \cdot n}$. Similarly, the $\ell_1$-regularizer assumes that the entries of $S$ are drawn iid from the Laplace distribution with mean 0 and variance estimated as $\hat{b} = \frac{\|S\|_1}{m \cdot n}$. The nuclear norm of $L$ is the $\ell_1$-norm of its singular values, so its corresponding prior assumes that the singular values of $L$ follow the Laplace distribution. The log-likelihood function is then
\begin{multline}
\ell(L,S,X) = -\left( \frac{m \cdot n}{2} \right) \log(2 \pi \sigma^2) - \frac{\|L+S-X\|_F^2}{2 \sigma^2} -\left( m \cdot n \right) \log(2 b) - \frac{\|S\|_1}{2 b} - \rank(L) \log(2 b_*) - \frac{\|L\|_*}{2 b_*},
\end{multline}
where $\sigma^2,b$ and $b_*$ are computed according to their respective estimator. We must also define the number of parameters of this model, which is equal to its degrees of freedom. Each term in SPCP provides the following degrees of freedom:
\begin{align*}
\rank(L) = k \quad &\rightarrow  \quad k(m+n-k) \textnormal{  degrees of freedom}, \\
\|S\|_1 \quad &\rightarrow \quad nnz(S) \textnormal{  degrees of freedom}, \\
\tfrac{1}{2} \|L+S-X\|_F^2 \quad &\rightarrow \quad \left( \tfrac{\|L+S-X\|_F^2}{\|X\|_F^2} \right) (m \cdot n) \textnormal{    degrees of freedom}.
\end{align*}
Note that $nnz(S)$ counts the number of non-zero entries in $S$. Finding the degrees of freedom in a rank-$k$ matrix or a sparse matrix are both standard calculations. For the loss term, we use the residual effective degrees of freedom as an estimate of the flexibility that it introduces. This is detailed further in \cite{LoaderRegression}. The number of parameters $p$ is equal to the total degrees of freedom.

The AIC$_{\textnormal{\large c}}$ values for the solutions shown in Figure \ref{Pix} are listed in Table \ref{Table:DoF}. The AIC$_{\textnormal{\large c}}$ value for the reference solution is listed under ``oracle.'' Recall that the reference solution was found using the LagQN solver with a tight tolerance. All the other values in the table correspond to solutions that meet comparable tolerances. Also, since the IALM model does not use an $\ell_1$-norm regularizer, many of the values in the returned $S$ matrix were small but not exactly zero. To make accurate comparisons, we applied the shrinkage operator to the $S$ matrix returned by IALM to set small values equal to zero before calculating the degrees of freedom. The values for IALM and LMaFit are high because these solvers are not generally robust to large amounts of noise.
\begin{figure}[h!]
\centering
\begin{tabular}{c c}
\centering
Solver &  AIC$_{\textnormal{\large c}}$ ($\times 10^6$) \\ \hline
{\it Oracle } & ${\it 7.37}$ \\ 
{\bf Split-SPCP} & ${\bf 7.64}$ \\ 
GoDec & 10.95 \\
LagQN & 11.30 \\
MaxQN & 11.64 \\ 
FPCP & 15.70 \\ 
IALM & $2.43 \times 10^7$ \\ 
LMaFit & $3.99 \times 10^9$ \\ 
\end{tabular}
\captionof{table}{Degrees of freedom in solutions returned by various solvers.} 
\label{Table:DoF}
\end{figure}
With the AIC$_{\textnormal{\large c}}$ metric as well, we see that Split-SPCP discovers the best solution, and in a much shorter time compared to the other algorithms.

Split-SPCP has an advantage because it parallelizes well on GPU architectures. This is due to the relative simplicity of the L-BFGS algorithm, which avoids complex linear algebraic decompositions that require heavy communication, such as the QR and SVD (see also~\cite{ParallelLBFGS}). Figure \ref{fig:GPU} shows that although Split-SPCP on the CPU is slower than LagQN on the CPU, 
Split-SPCP enjoys enormous speedup when it is implemented on the GPU.

\begin{figure}[h!]
\centering
\includegraphics[width=0.6\linewidth]{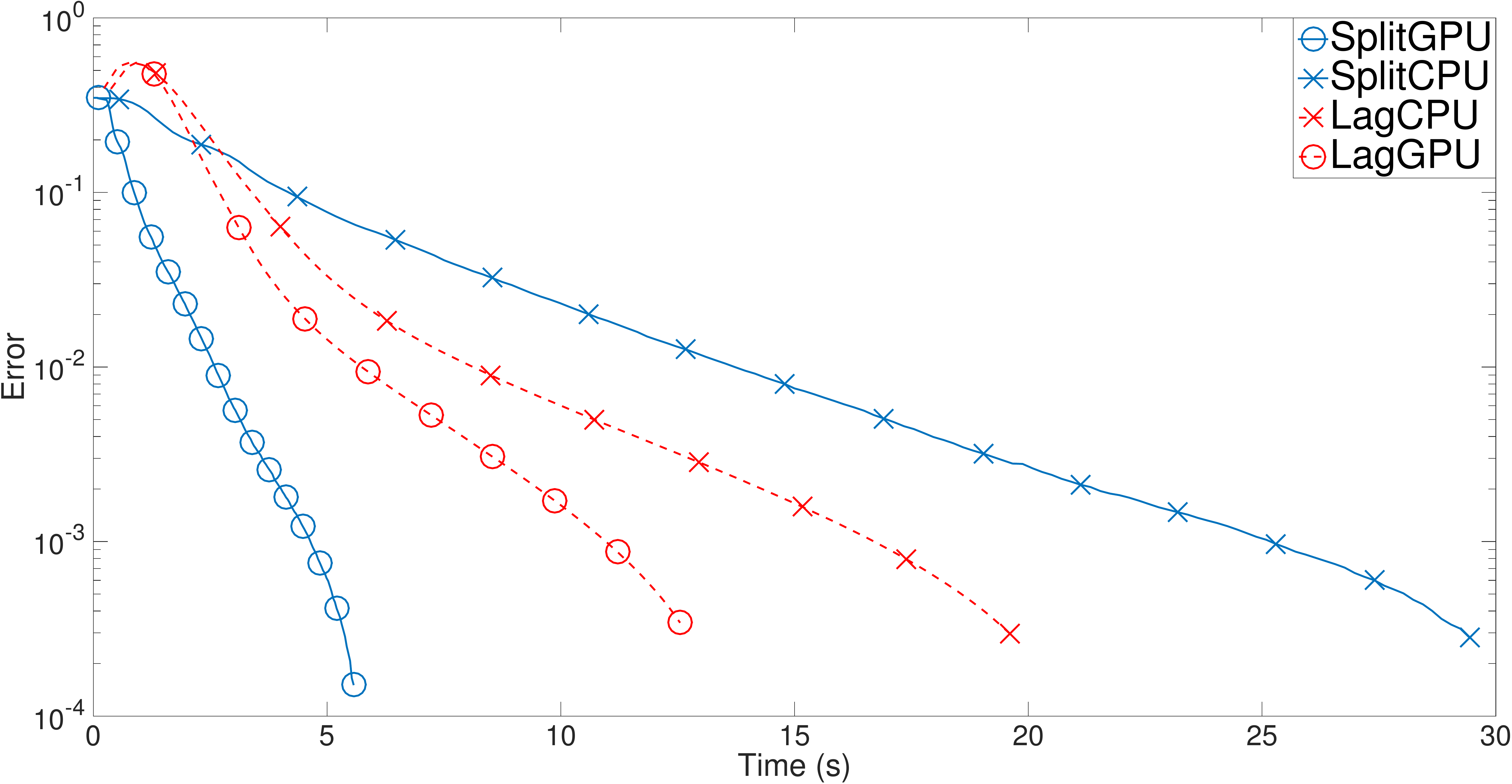}
\caption{Comparing the performance of Split-SPCP and LagQN on the CPU and the GPU. Although Split-SPCP is slower than the similar convex solver on the CPU, it sees more acceleration on the GPU because it avoids QR and SVD computations. Tests were performed on the synthetic data described in Section \ref{SynDat}}
\label{fig:GPU}
\end{figure}

\subsection{Synthetic Data}
\label{SynDat}

Split-SPCP demonstrates similar performance on synthetic data. For our synthetic-data test, we create two $1,\!000 \times 150$ random matrices with each entry drawn from independently from the univariate normal distribution, and define our low-rank reference matrix, $L_{ref}$, as the product of these two matrices. To form the sparse reference matrix, $S_{ref}$, we fill 50\% of a $1,\!000 \times 1,\!000$ matrix with numbers drawn independently from the univariate normal distribution, and the other entries we set equal to zero. $X$ is then the sum $L_{ref}+S_{ref}$ with added white Gaussian noise, so that $\frac{\| L_{ref}+S_{ref}-X \|_F}{\|X\|_F} = 8.12 \times 10^{-5}$. Matrices with these characteristics are often encountered in video processing applications.

We initialize Split-SPCP with the first 200 singular value and vector triples of $X$, \textcolor{black}{so that $k$ is larger than the rank of $L$}. These triples are found using the rSVD. The tuning parameters are set to $\lambda_L = 2.95$ and $\lambda_S = 0.1$. As before, we measure performance based on the normalized difference from the true Split-SPCP objective value. The results are shown in Figure \ref{Synth}.

\begin{figure}
\centering
\includegraphics[width=0.6\linewidth]{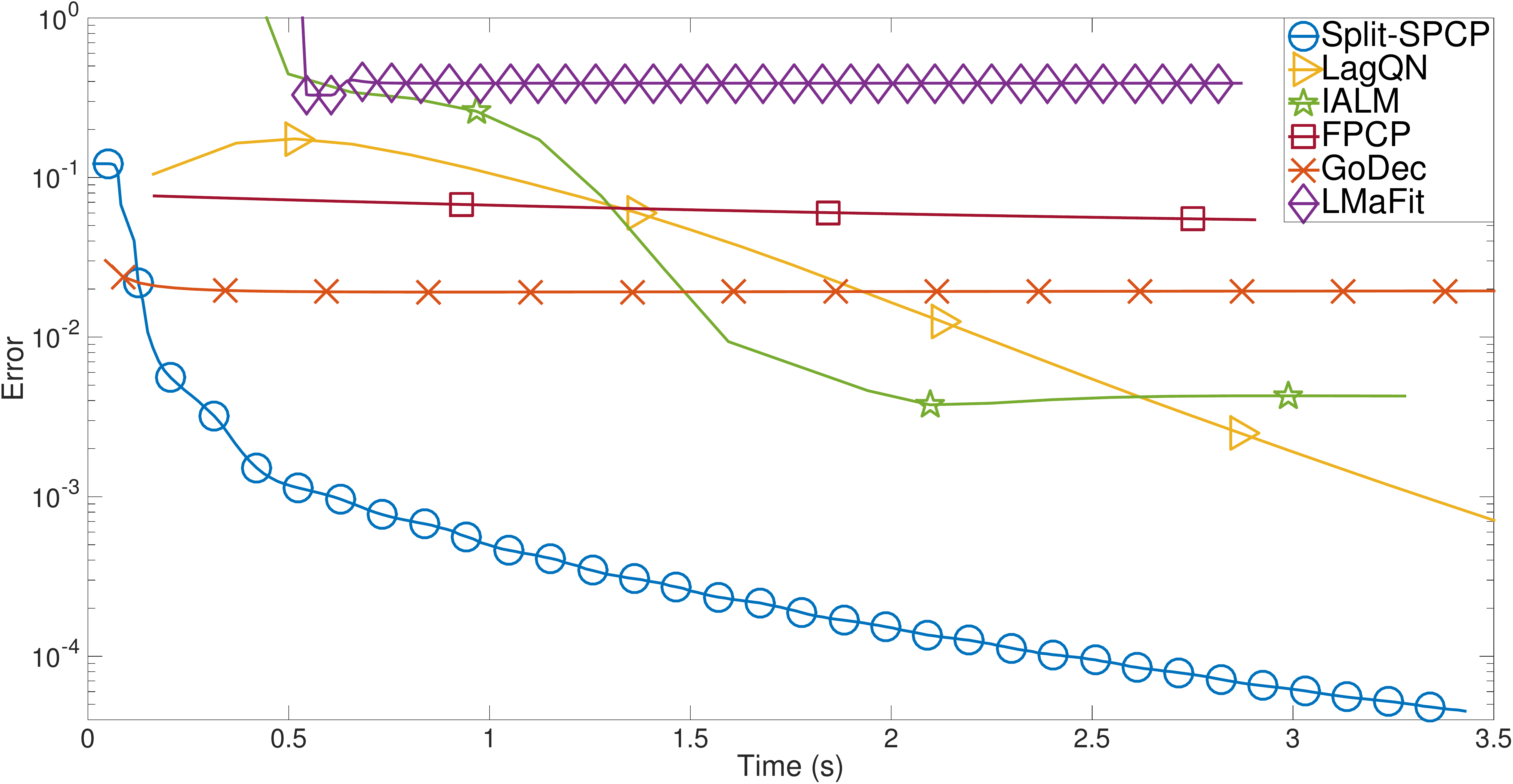}
\caption{Performance of SPCP solvers on the synthetic data described in Section \ref{SynDat}, on the GPU. One marker represents six iterations.}
\label{Synth}
\end{figure}

We see that FPCP and GoDec discover approximate solutions quickly, but the approximations are not within $10^{-2}$ of the true objective. LMaFit exhibits poor performance. This could be due to the fact that $S$ is about 50\% sparse, and LMaFit struggles with problems that have a large sparse component~\citep{LMaFit}. Split-SPCP significantly outperforms all solvers in speed, and outperforms all solvers except for LagQN in accuracy. LagQN and Split-SPCP converge to solutions with the same objective value, albeit Split-SPCP converges much more quickly.

\subsection{Dimensional Scaling and Real-Time Video Processing}
\label{DimScale}

Background subtraction in video is a quintessential application for low-rank recovery algorithms, but these algorithms are generally unable to perform real-time background subtraction due to their slow implementation. The experiments in Section \ref{subsec:BackSubtract} show that Split-SPCP is faster and more accurate than the methods reviewed in \cite{ThierryVideoReview} and \cite{Sobral} for background subtraction, and in this section, we use Split-SPCP for the real-time processing of high-resolution videos with standard frame rates. 
We also explore how Split-SPCP scales as the resolution and frame rate increases.

To measure the performance of Split-SPCP on more diverse video datasets, we use videos from the Background Model Challenge \citep{BMC}. The experiments shown in Figure \ref{Fig:col} consider a video with $240 \times 320$ pixels per frame and 32,964 frames, which is substantially larger than the low-resolution escalator video of Section \ref{subsec:BackSubtract} with $130 \times 160$ pixels per frame and only 200 frames. 
We unfold the three-dimensional video data into a matrix with each column containing one frame, and we partition the matrix columnwise into blocks comprising a fixed number of frames. 
With the rank bound fixed at 60, we record the computation time required to perform 200 iterations of Split-SPCP on one block as a function of the block size.
Larger block sizes yield better solutions, but they also require more time for SPCP to converge.
We find that after 200 iterations, Split-SPCP converges to within $10^{-3}$ of the optimal objective value for all block sizes, 
and the solution is qualitatively similar to the optimal solution. 

Figure \ref{Fig:col} shows how Split-SPCP scales as the video resolution and the number of frames per block increase. 
These two quantities correspond to the number of rows and the number of columns of each block. 
We increase the resolution by rescaling each frame so that the aspect ratio stays close to constant, 
and we use interpolation to impute the missing pixel values. 
It is easy to discern the $\mathcal{O}(n)$ scaling as the number of frames (columns) is increased. 
Increasing the resolution affects the computation time linearly, but beyond a ``critical resolution'' of $332 \times 446$ pixels per frame, 
the slope of the linear scaling undergoes a shift in regime. This linear scaling is expected because the bottleneck of Split-SPCP is the matrix-matrix multiply $UV^T$, 
which requires $\mathcal{O}(mnk)$ operations, so with $k$ fixed, the computation time of Split-SPCP grows linearly with the number of data points. 
The results of the background subtraction are shown in Figure \ref{BackSubEx}.

\begin{figure}
\centering
\includegraphics[width = 0.4\linewidth]{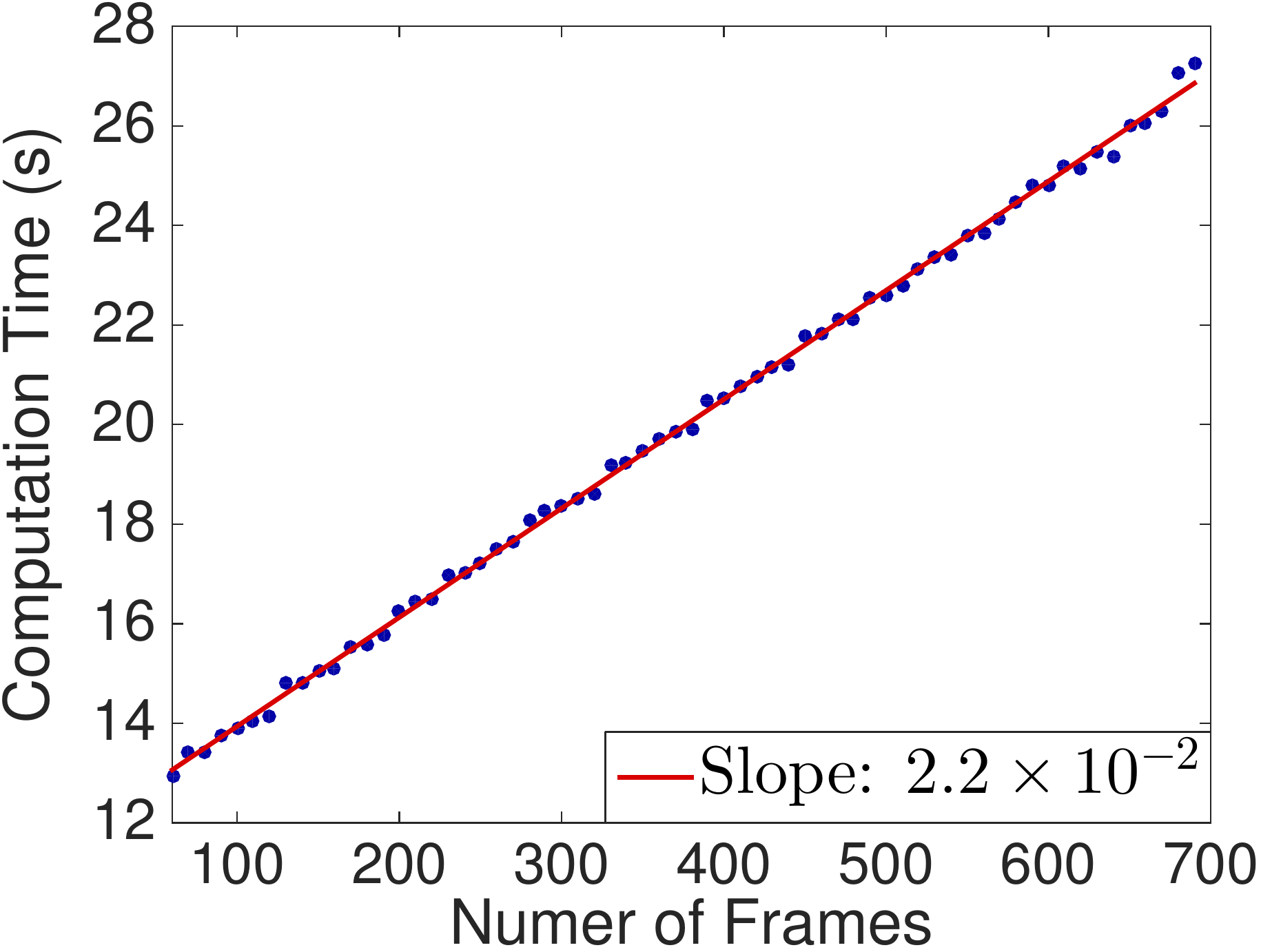}
\includegraphics[width = 0.4\linewidth]{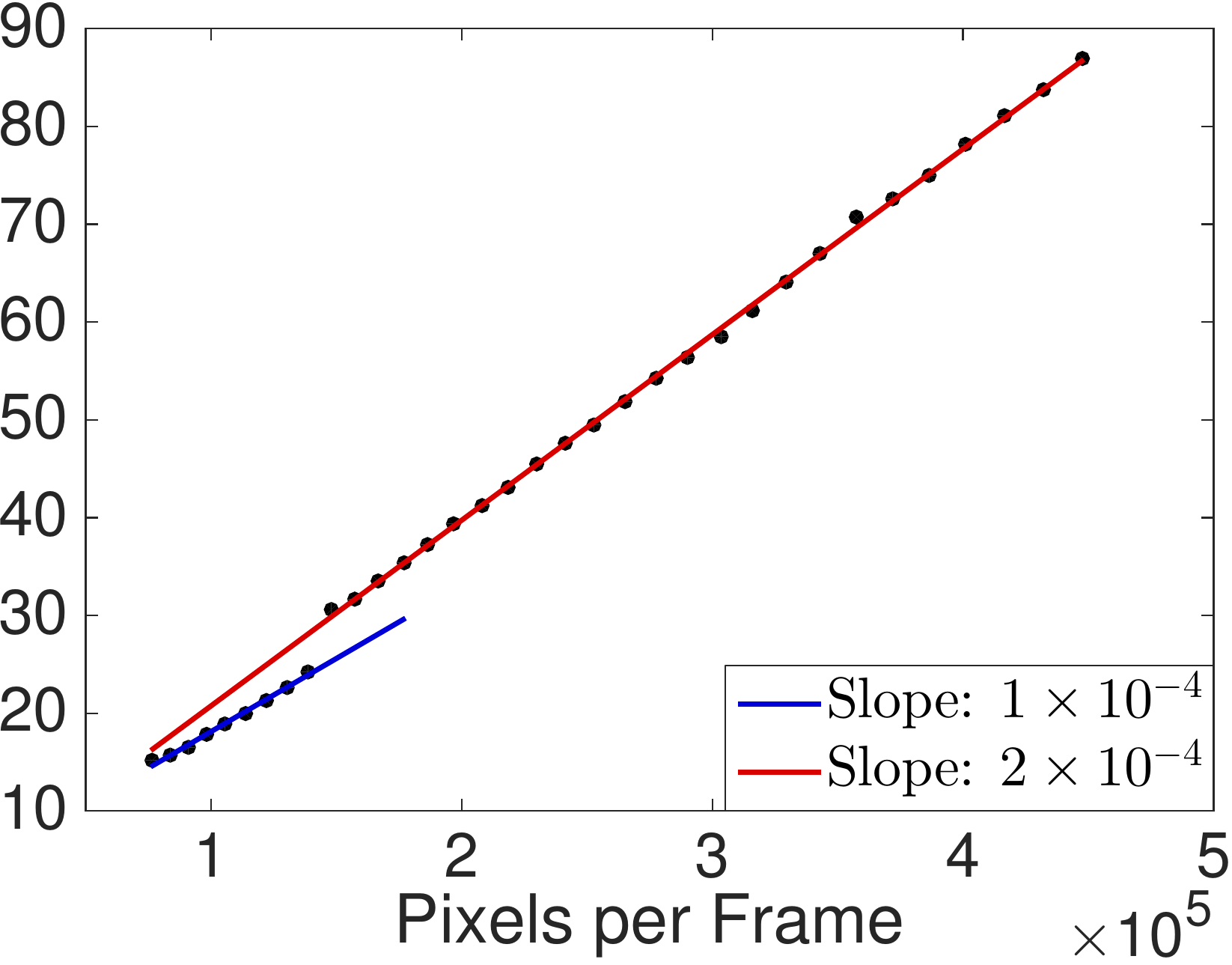}
\caption{The scaling of Split-SPCP as the number of columns (or frames) increases (left) and as the number of rows (or pixels per frame) increases (right). Each point is the average of twenty trials. For the column-scaling tests, the resolution is fixed at $240 \times 320$ pixels per frame, and for the row-scaling tests, 
the number of frames in a block is fixed at 100. The tuning parameters are set to $\lambda_L = 200, \lambda_S=5$.}
\label{Fig:col}
\end{figure}

These experiments suggest conditions under which Split-SPCP can decompose a video in real-time. Assuming a resolution of $240 \times 320$ pixels per frame and a frame rate of 24 frames-per-second, which is standard in many applications of video processing, the video must be partitioned blocks of about 600 frames or more. With this partitioning, the algorithm will have finished decomposing the given block before the next block is recorded. This number increases linearly with resolution, following the trend line given in the rightmost plot of Figure \ref{Fig:col}. Following these guidelines, Split-SPCP can decompose the videos from \cite{BMC} in real-time without sacrificing quality. An example is shown in Figure \ref{BackSubEx}.

\begin{figure}[h!]
\includegraphics[width=0.3\linewidth]{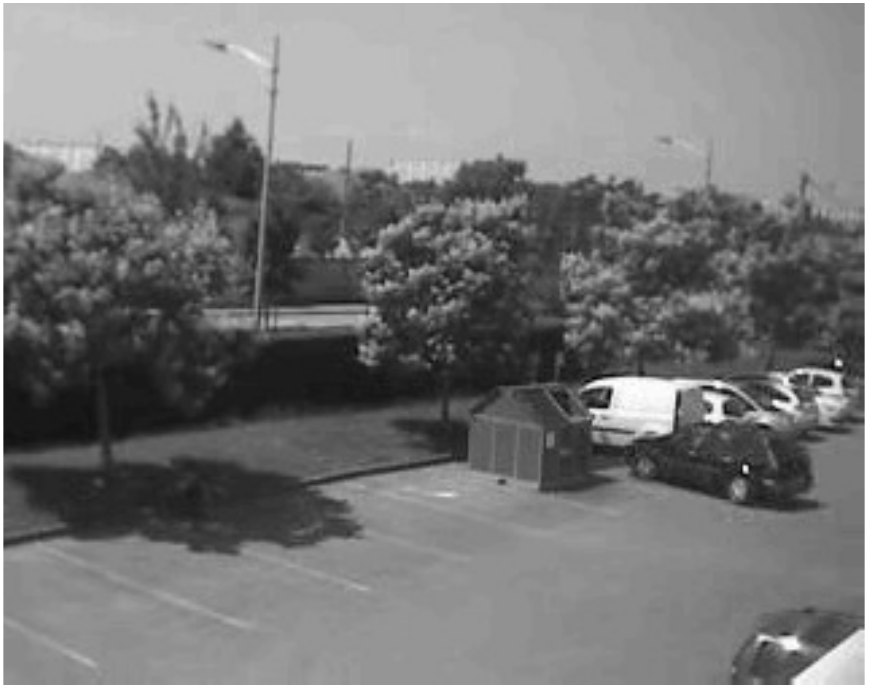}
\includegraphics[width=0.3\linewidth]{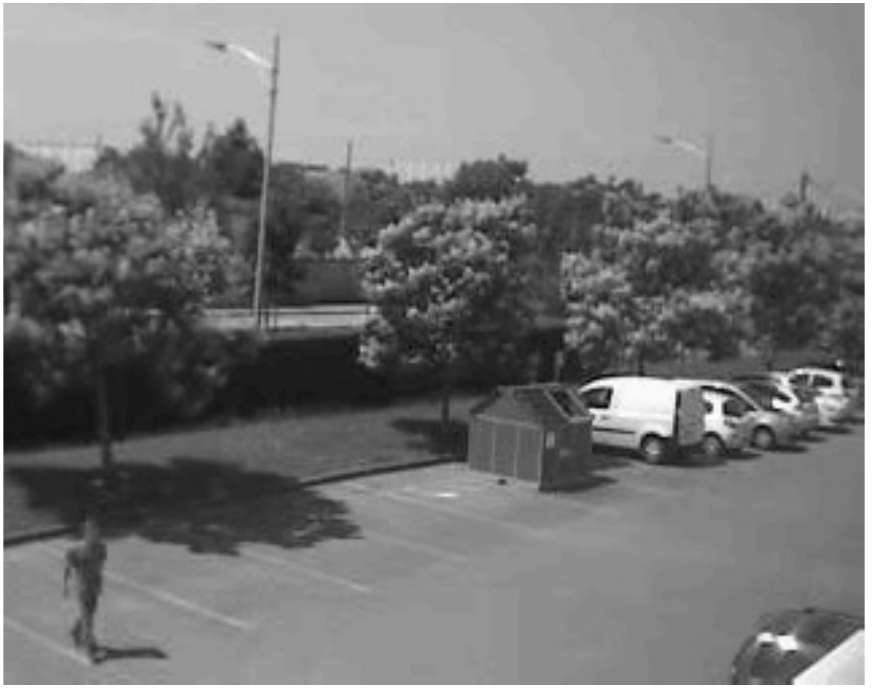}
\includegraphics[width=0.3\linewidth]{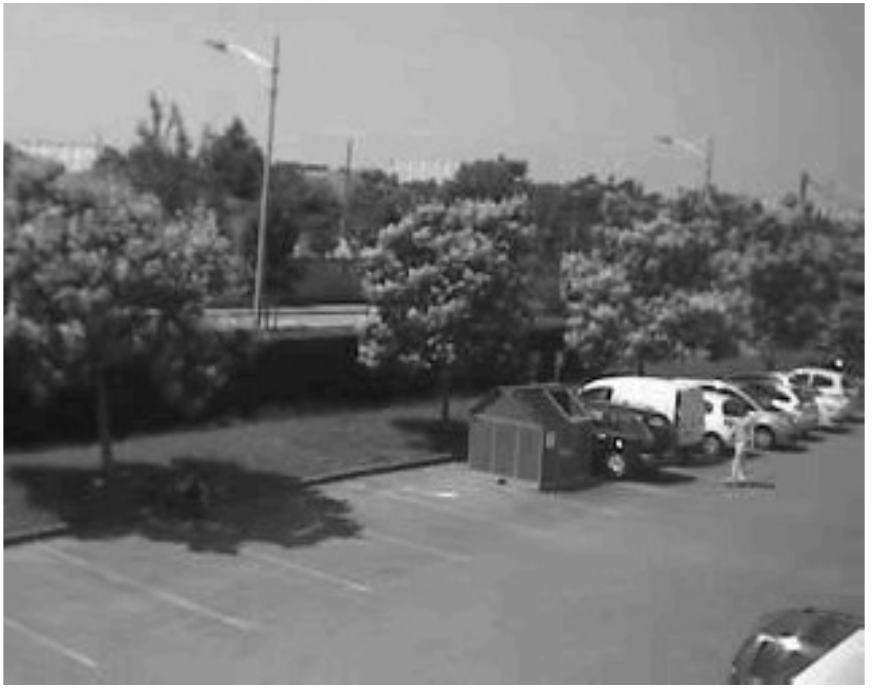}
\includegraphics[width=0.3\linewidth]{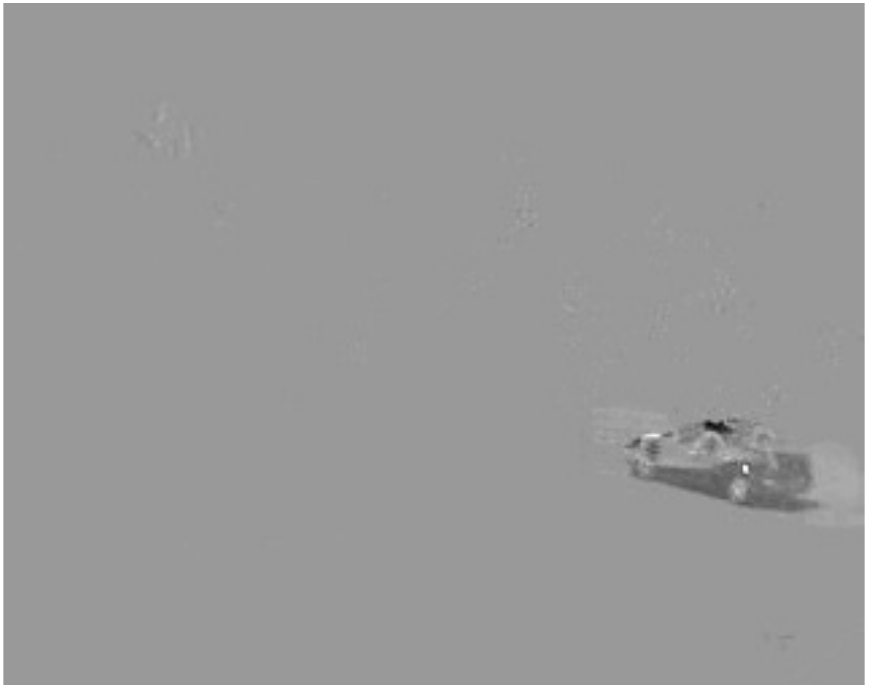}
\includegraphics[width=0.3\linewidth]{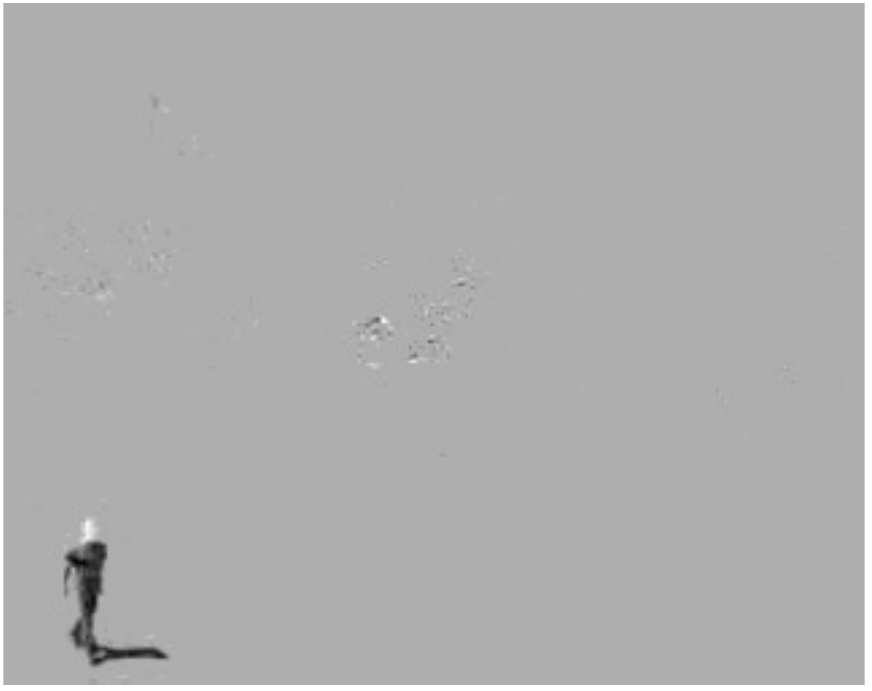}
\includegraphics[width=0.3\linewidth]{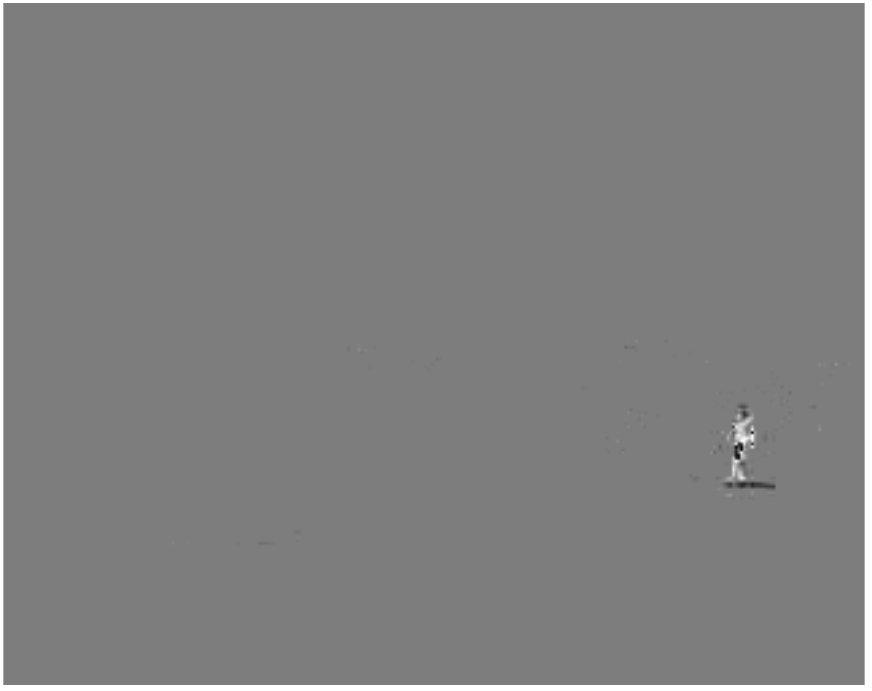}
\caption{Real-time background subtraction using Split-SPCP. Split-SPCP correctly identifies (left) a car in the middle-ground, (center) a person in the foreground, and (right) a person in the background. The resolution is $240 \times 320$ pixels per frame, 600 frames were decomposed at one time, and only 200 iterations were allowed.}
\label{BackSubEx}
\end{figure}

\section{A Comparison with Accelerated-SVD Algorithms}
\label{sec:fastSVDcompare}

There are several algorithms for low-rank matrix recovery that focus on accelerating the SVD step. In this section, we explore two such approaches: a communication-avoiding rSVD and Frank-Wolfe with marginalization. We compare both to our SVD-free approach. Both of these methods preserve convexity, but numerical experiments suggest that our non-convex formulation is superior.

\subsection{Communication-Avoiding rSVD}
\label{tsqr}

Many problems in data analysis involve matrices that are ``tall-skinny''. For example, the surveillance video we consider in Section \ref{BurMontEx} has dimensions $20,\!800 \times 200$, and the psychometric data shown in Figure \ref{fig:fMRI} produces a matrix with 230-times more rows than columns. 
The prevalence of this tall-skinny structure is due in part to the fact that multi-dimensional data sets must be ``unwrapped'' into a two-dimensional matrix before matrix-based algorithms can analyze the data, and the canonical unwrapping of a square data set with more than two dimensions leads to a highly rectangular matrix. Tall-skinny matrices pose a challenge for matrix decompositions, 
especially on computational architectures where communication is costly (such as a GPU)~\citep{TSQR}. 

The rSVD is often used to accelerate low-rank recovery models by projecting a data matrix of size $m \times n$ onto a low-dimensional subspace, and performing computations with the resulting smaller matrix of size $m \times k$, with $k \ll n$.\footnote{Further details of the rSVD algorithms can be found in \cite{halko2011finding}} The rSVD offers improved performance on the CPU, but it exacerbates the tall-skinny problem, making parallelization difficult even when the input matrix is square.

The effect of this increased communication is notable: on the CPU, the rSVD step in LagQN accounts for about 55\% of the algorithm's total running time, but when this solver is run on the GPU, the rSVD accounts for about 90\% of SPCP's total running time, as shown in Figure \ref{fig:PsychSPCP}. 
The TSQR algorithm of \cite{TSQR} can provide a convex option for accelerating low-rank recovery models. 
While TSQR has been used to accelerate the SVD step in RPCA for tall-skinny matrices \cite{AndersonTSQR}, 
we extend this idea to the rSVD algorithm on general matrices, and we compare this communication-avoiding approach to Split-SPCP for analyzing fMRI brain scan data.

For the experiments in this section, we implement the TSQR algorithm on a single GPU card as well as multiple CPUs. The GPU card and CPUs are the same models described in Section \ref{subsec:ImpDets}. We use 32 CPU cores: two sockets, eight cores per socket, and two threads per core. To run QR decompositions in parallel on the GPU, we use the batched QR algorithm of MAGMA version 1.7.0 \cite{Magma}, compiled as a MEX file using GCC for Red Hat version 4.8.5-11. We wrote our own implementation of TSQR on multiple CPUs using MATLAB's Parallel Computing Toolbox. The initial number of blocks in our TSQR implementation is 32, corresponding to the 32 CPU cores we have available. Specific details on the TSQR algorithm can be found in \cite{TSQR}.

fMRI brain scan data sets are inherently four-dimensional, so unwrapping the dataset into a two-dimensional array creates a tall-skinny matrix with hundreds of times more rows than columns. The enormity of these data sets makes low-rank recovery models intractable without parallelization. Figure \ref{fig:PsychSPCP} shows the total time spent performing SPCP on one of these data sets, as well as the total time spent performing rSVDs, performing the QR step of the rSVD, and moving data (labeled as the ``overhead'' cost). The results of this test are included in Figures \ref{fig:fMRI} and \ref{GlobalSig}.

\begin{figure}
\centering
\includegraphics[width=0.6\linewidth]{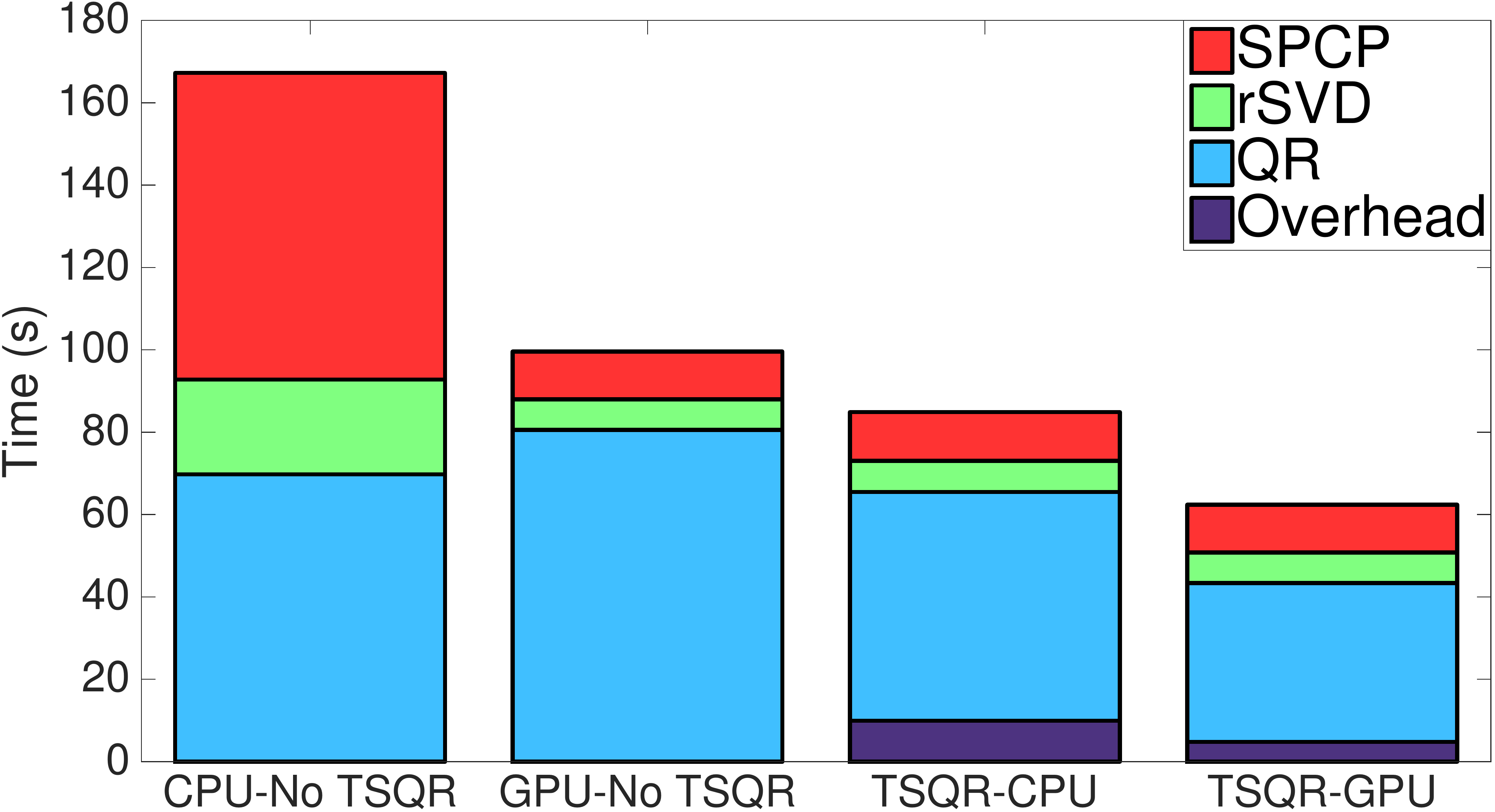}
\caption{The time spent performing subprocesses of the SPCP algorithm on fMRI brain scan data. Lower blocks are subprocesses of upper blocks, so the times are cumulative. The test ``TSQR-CPU'' uses a hybridized architecture, running TSQR on multiple CPU's but running the rest of SPCP on the GPU. Both tests ``GPU-No TSQR'' and ``TSQR-GPU'' were run entirely on the GPU.}
\label{fig:PsychSPCP}
\end{figure}

\begin{figure}[h!]
\centering
\includegraphics[width=0.3\linewidth]{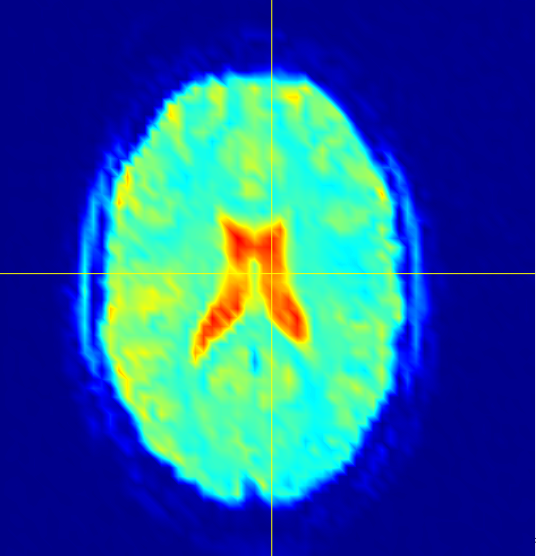}
\includegraphics[width=0.3\linewidth]{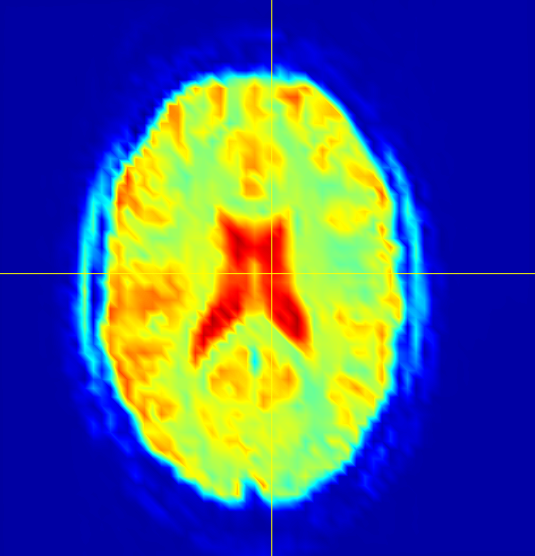}
\includegraphics[width=0.3\linewidth]{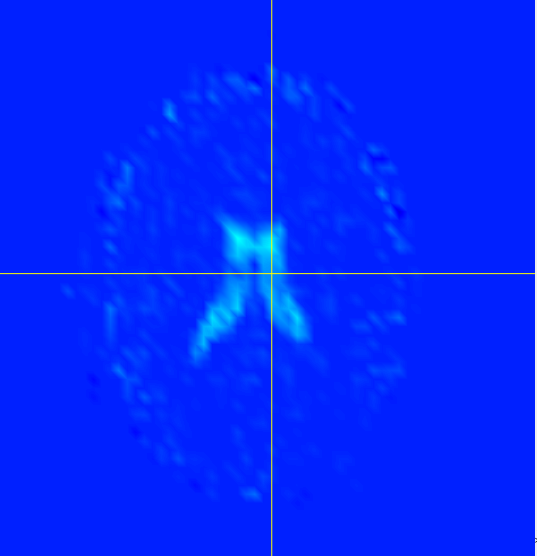}
\caption{The results of SPCP on fMRI brain scan data. Activity is measured in Blood-oxygen-level dependent (BOLD) signal. (From left to right: original image, low-rank image recovered by SPCP, and sparse image recovered by SPCP.)}
\label{fig:fMRI}
\end{figure}

The datasets used for testing were taken from a study of the human brain's response to uncomfortable stimuli, specifically, the feeling of submerging a hand in hot or cold water. Analyzing these scans to find common neurological responses to these stimuli is difficult due to the enormous amount of error in the data. There is uniformly distributed noise due to constant ancillary physiological activity, and there are also sparsely distributed groups representing neurological structures that should all exhibit the same behavior. The ventricles, for example, are filled with cerebrospinal fluid (CSF), which does not contribute to neurological communication, so they should not be active. All signals observed in the ventricles should be treated as sparsely structured outliers.
SPCP removes the uniform noise and, most remarkably, correctly identifies signals in the brain's ventricles as outliers. In Figure \ref{fig:fMRI}, the largest ventricles are the two structures in the center of the brain. The rightmost image shows that the majority of the noise contained in $S$ is from these ventricles. 

The other two major components of the brain are white and gray matter. The activity we are hoping to observe takes place in the gray matter, so ideally SPCP would remove most signals from the white matter regions. However, the regions of white matter are more difficult to distinguish than the regions of CSF, and SPCP removes about equal amounts of noise from the white matter as it does from the gray. If we let $S_{gm}$ be the gray-matter component of $S$, and define $S_{wm}$ and $S_{csf}$ similarly, Figure \ref{GlobalSig} shows the average BOLD signal in $S_{gm}, S_{wm}$, and $S_{csf}$ for each frame in time. These data were normalized by the average original signal in the corresponding regions.

\begin{figure}[h!]
\includegraphics[width=0.332\linewidth]{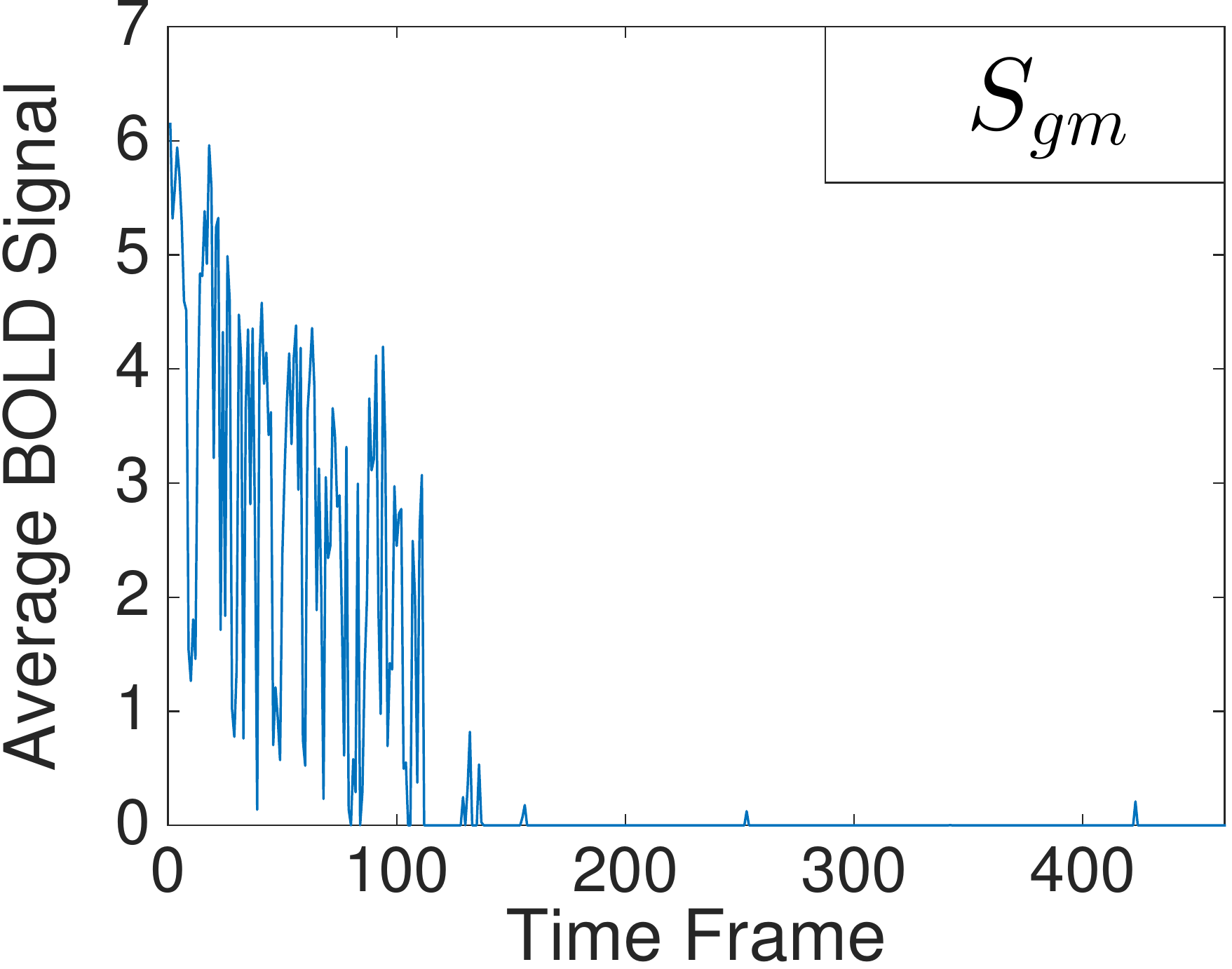}
\includegraphics[width=0.3\linewidth]{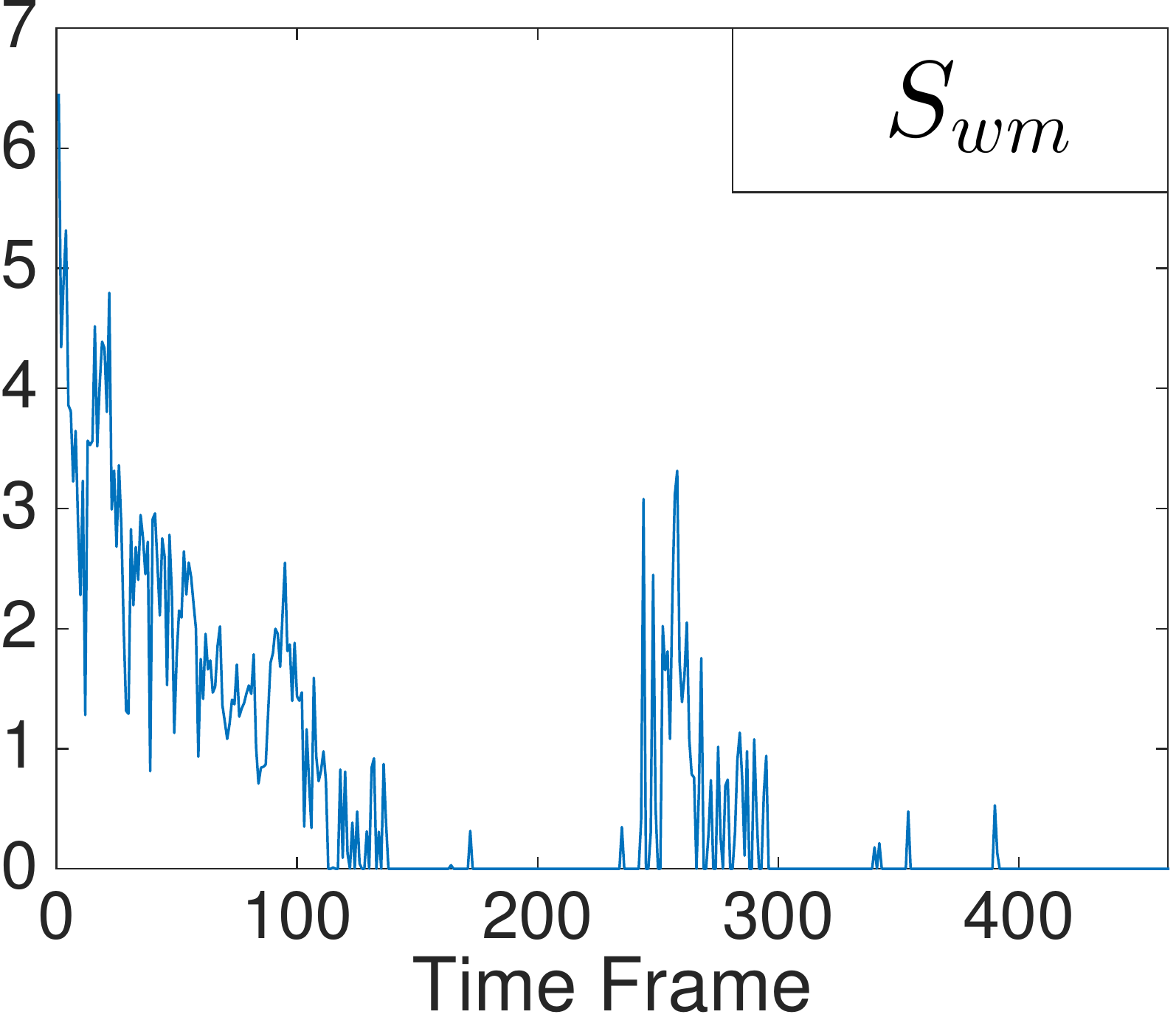}
\includegraphics[width=0.3\linewidth]{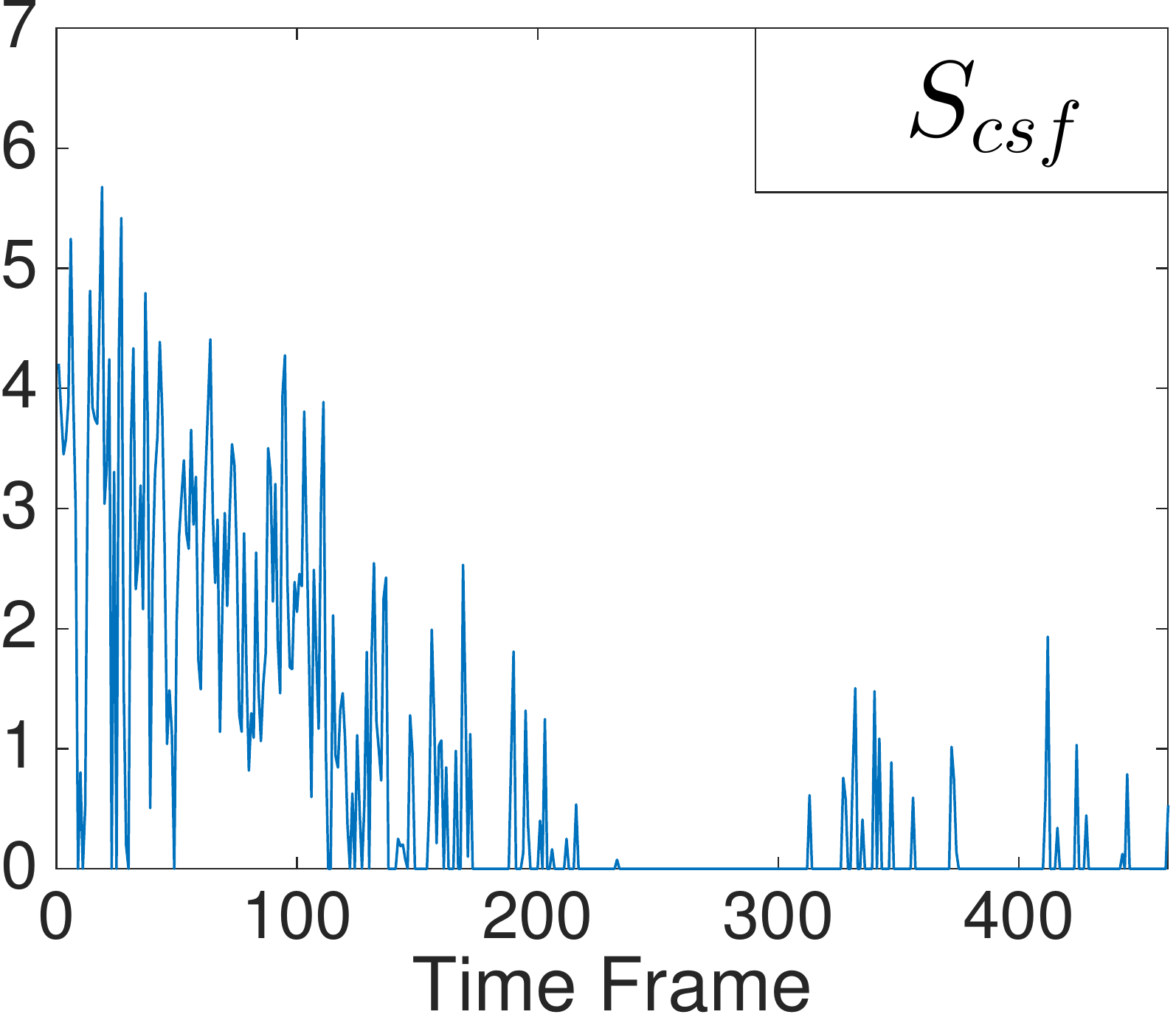}
\caption{The average BOLD signal in different regions of $S$. The averages were normalized by the average original signal in the corresponding regions.}
\label{GlobalSig}
\end{figure}

It is clear that $S_{csf}$ contains more signal than the other two regions. For $S_{gm}$, SPCP detects noise in only the first 100 time slices. The removed signal from the white matter is more distributed over time, and the total amount of noise in $S_{gm}$ and $S_{wm}$ is comparable. These results suggest that SPCP correctly identifies outliers in the fMRI data, especially within the regions of CSF.

Although using TSQR to reduce communication costs provides noticeable speedup, the non-convex approach is faster. In Figure \ref{Compare}, we compare our non-convex Split-SPCP algorithm to a convex solver using TSQR for the rSVD step. Both solvers were decomposing the same $106,\!496 \times 462$ brain-scan data-matrix used in section 6.4, and both were run on the GPU.

\begin{figure}[h!]
\centering
\includegraphics[width=0.7\linewidth]{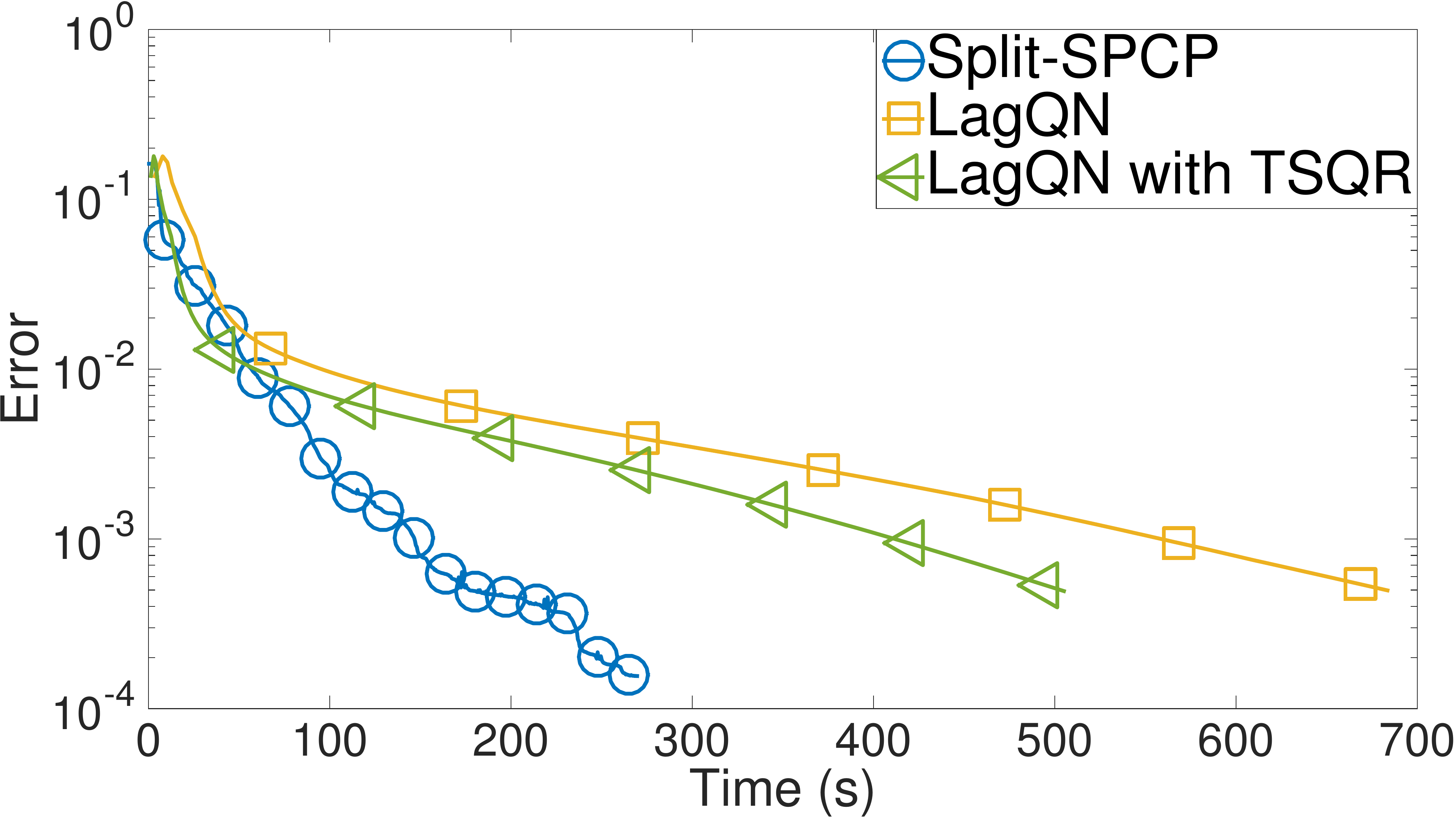}
\caption{Comparing Split-SPCP to LagQN using TSQR on the GPU and without using TSQR. Split-SPCP was initialized with the rank bound $k = 462$, the maximum possible rank of the low-rank component. One marker represents 35 iterations.}
\label{Compare}
\end{figure}

As with previous tests, a reference solution $(L_{ref},S_{ref})$ was found by solving LagQN to high accuracy, and the error is measured as the normalized distance from the reference objective value. The rank of $L_{ref}$ was 423, and Split-SPCP was initialized with the rank bound $k = 462$. Since $X$ had 462 columns, this was the largest possible rank bound, and Split-SPCP still greatly outperformed the convex solvers. Both solvers also found solutions of similar quality. At an error of $5 \times 10^{-4}$, the low-rank component found by LagQN had rank 427, and the low-rank component found by Split-SPCP had rank 425. Similarly, the sparse component of LagQN had sparsity $58.5\%$, and the sparse component of Split-SPCP had sparsity $58.3\%$, while $S_{ref}$ was $58.5\%$ sparse. Split-SPCP recovered nearly the same solution as the parallelized convex algorithm at a much smaller cost.

\subsection{Frank-Wolfe with Marginalization}
\label{AFW}

Our approach to accelerating low-rank recovery models is closely connected to the one presented in \cite{FWforRPCA} to accelerate the Frank-Wolfe method for RPCA. The Frank-Wolfe method, first proposed in \cite{FW}, has recently gained attention for its ability to yield scalable algorithms for optimization problems involving structure-encouraging regularizers \citep{JaggiFrankWolfe, FWforRPCA}. Frank-Wolfe methods optimize a Lipschitz-differentiable function over a convex set by solving a sequence of linear subproblems. For example, if $f$ is Lipschitz-differentiable and $\mathcal{C}$ is a convex set, Frank-Wolfe solves the problem
\begin{equation}
\min_{x \in \mathcal{C}} \quad f(x)
\label{eq:FW}
\end{equation}
by linearizing $f$ about each iteration $x_k$: $f(y) \approx f(x_k) + \big< \nabla f(x), y-x\big>$, minimizing the linear problem to find a decent direction $y_k-x_k$, and then stepping in this direction with a step size $\eta$: $x_{k+1} = x_k + \eta (y_k - x_k)$. Traditionally, $\eta = \frac{2}{k+2}$ is fixed for convergence guarantees, but better performance can be achieved with adjustable step sizes \cite{FWforRPCA}.

For optimization problems involving the nuclear norm, the set $\mathcal{C}$ is a scaled nuclear norm ball, and the linear subproblem becomes a search for the leading singular tuple of $L_k$. This is much cheaper than computing the full SVD at each iteration, which is required by proximal methods.

In contrast to nuclear norm minimization problems, it is often difficult to scale the Frank-Wolfe method for problems involving $\ell_1$ regularization. For these problems, each Frank-Wolfe iteration updates iterate $S_k$ with only a single non-zero element \citep{FWforRPCA}. This makes Frank-Wolfe infeasible when the sparse component of the data-matrix has several thousand non-zero elements or more. The authors of \cite{FWforRPCA} adapt the Frank-Wolfe method to solve SPCP, using the traditional Frank-Wolfe update for the low-rank component and using a projected gradient update for the sparse component. Their adaptations allow for significant speedup and better convergence compared to the traditional Frank-Wolfe.

The techniques used in \cite{FWforRPCA} can be seen as a special case of the marginalization we present in this work. With the sparsely structured component marginalized, the program depends only upon the low-rank component, so the benefits of the Frank-Wolfe scheme are preserved and its drawbacks are negated. Algorithm \ref{Al:FWours} shows how marginalization can be used to extend the adaptations presented in Algorithms 5 and 6 of \cite{FWforRPCA}. Algorithm \ref{Al:FWours} solves the following program, which is equivalent to SPCP:
\begin{equation}
\min_{\|L\|_* \le t} \quad \lambda_L t + \min_S \frac{1}{2} \| \mathcal{P}_{\Omega} (L+S-X) \|_F^2 + \lambda_S \|S\|_1.
\end{equation}
With the $S$ variable marginalized, it does not affect the linear subproblems arising in the Frank-Wolfe scheme, making this approach much more scalable. Each iteration of Algorithm \ref{Al:FWours} updates the low-rank component and the nuclear norm bound, $t$, while the sparse component remains only implicitly defined. More details on this approach can be found in \cite{FWforRPCA}. 

While Algorithm \ref{Al:FWours} is similar to the methodology in \cite{FWforRPCA}, fully marginalizing the sparse component eliminates the linear subproblem required in \cite{FWforRPCA} to update the iterate $S_k$. We show the performance of this adapted Frank-Wolfe scheme on SPCP in Figure \ref{FWspcp}. Although its performance is better than the traditional Frank-Wolfe (see \cite{FWforRPCA}), it is still much slower than solving Split-SPCP with L-BFGS.

\begin{figure}[h!]
 \includegraphics[width=0.5\linewidth]{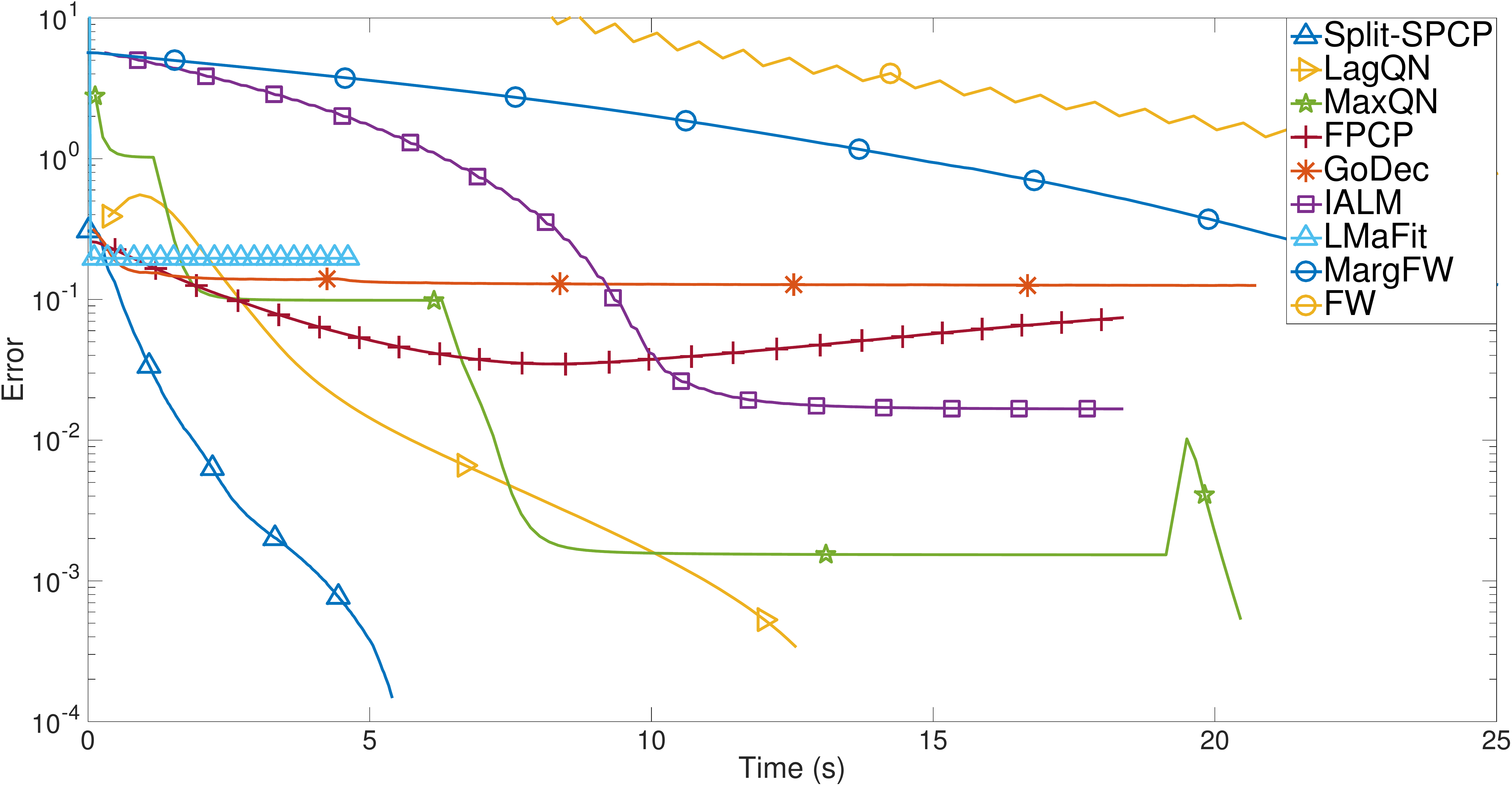}
\caption{Comparing Frank-Wolfe with marginalization (denoted ``MargFW''), as well as another Frank-Wolfe scheme adapted for RPCA \cite{FWforRPCA} (denoted ``FW''), with other SPCP solvers for background subtraction on the escalator video of \cite{LiData}.  One marker corresponds to 40 iterations.}
\label{FWspcp}
\end{figure}

\begin{algorithm}
\DontPrintSemicolon
\caption{Frank-Wolfe Method with Marginalization}
\label{Al:FWours}
\SetKwInOut{Input}{Input}
\SetKwInOut{Output}{Output}
\Input{$\mathbf{L_0}, \lambda_L, \lambda_S, t_0 = 0, U_0 = \frac{1}{2 \lambda_L} \|\mathcal{P}_{\Omega} (\mathbf{L_0}+\mathbf{S_0}-\mathbf{X}) \|_F^2$}
\While{\textnormal{Not Converged}}{
	$\mathbf{S_k} \leftarrow \textnormal{shrink}(\mathbf{X} - \mathbf{L_k}, \lambda_S)$\;
	$\mathbf{Y_k} \leftarrow \argmin\limits_{\|\mathbf{Y}\|_* \le 1} \big< \mathcal{P}_{\Omega} (\mathbf{L_k} + \mathbf{S_k} - \mathbf{X}), \mathbf{Y} \big>$ \;
    \If{$\lambda_L \ge -\big< \mathcal{P}_{\Omega} (\mathbf{L_k} + \mathbf{S_k} - \mathbf{X}), \mathbf{Y_k} \big>$}{
    	$\mathbf{V_k} \leftarrow \mathbf{0}$ \;
        $V_{t_k} \leftarrow 0$ \;}
    \Else{
    	$\mathbf{V_k} \leftarrow U_k \mathbf{Y_k}$ \;
        $V_{t_k} \leftarrow U_k$ \;}
    $\eta^{\star} \leftarrow \argmin\limits_{\eta \in (0,1)} \quad \frac{1}{2} \| \mathcal{P}_{\Omega}( (1-\eta) \mathbf{L_k} + \eta \mathbf{V_k}  + \mathbf{S_k} - \mathbf{X})\|_F^2 + (1-\eta) \lambda_L t_k + \eta \lambda_L t_k$ \;
    $\mathbf{L_{k+1}} \leftarrow (1-\eta^{\star}) \mathbf{L_k} + \eta^{\star} \mathbf{V_k}$ \;
    $t_{k+1} \leftarrow (1-\eta^{\star}) t_k + \eta^{\star} V_{t_k}$ \;
    $U_{k+1} \leftarrow \frac{1}{2 \lambda_L} \|\mathcal{P}_{\Omega} (\mathbf{L_0}+\mathbf{S_0}-\mathbf{X}) \|_F^2 + t_k$
    }
\end{algorithm}

\section{Conclusion}
\label{Conc}

In this manuscript, we present a general framework for accelerating regularized low-rank recovery models on parallel architectures. Factorizing the low-rank component $L = UV^T$ and optimizing over the factors $U$ and $V$ induces a low-rank constraint on $L$ while eliminating the iterative SVDs often required for nuclear norm minimization. Although this factorization sacrifices the convexity of the program, we develop a certificate to determine whether our model has converged to a minimum or to a spurious stationary point, 
showing empirically and theoretically that spurious stationary points do not pose a serious problem for our non-convex model. We also show that marginalizing the regularized $S$ variable creates an objective that is Lipschitz-differentiable when the loss function is strongly convex. The smooth objective allows the use of first-order solvers, which exhibit faster convergence and better parallelization.

We demonstrate the effectiveness of our framework on SPCP, showing that Split-SPCP runs an order of magnitude faster on the GPU than existing solvers. There are other regularized low-rank recovery processes that decrease the cost of the SVD step for improved efficiency, including communication-avoiding and Frank-Wolfe-type algorithms. To compare to Split-SPCP, we propose an improved algorithm in both classes, and show that our SVD-free approach offers greater speedup.

With the acceleration our framework provides, regularized low-rank recovery models achieve new applications. 
Split-SPCP can process video streams in real-time and decompose extremely large data sets, 
such as fMRI brain scans, whose analysis with low-rank recovery models would otherwise be infeasible.

\section{Acknowledgments}

We would like to thank Tor Wager and Stephan Geuter from the Neuroscience department of the University of Colorado-Boulder for the fMRI data and advice.

\bibliographystyle{abbrv}
\bibliography{thesisBecker}

\clearpage

\section*{Appendix A: Factorization Theorem Details}

Both problems share the same set of local minimizers. This was first observed
in the case of SDP in \cite{BurerMonteiro2005} and later generalized in \cite{Aravkin2014}.

Our variant of the theorem:
\begin{theorem}
Consider an optimization problem of the following form
\begin{equation} \label{eq:A1}
\min_{\X \succeq 0} \quad f(\X), \quad \textrm{such that}\quad \rank(\X)\le \rankVar
\end{equation}
where  $\X \in \mathbb{R}^{\dimVarX}$ is a positive semidefinite real matrix, and $f$ is a lower semi-continuous (lsc) function mapping to $[-\infty,\infty]$ and has a non-empty domain over the positive semi-definite matrices. 
Using the change of variable $\X = \P\P^T$, take $\P \in \mathbb{R}^{\dimVar \times \rankVar}$, and consider the problem   
\begin{equation} \label{eq:A2}
\min_{\P} \quad g(\P) \defeq f(\P\P^T)
\end{equation}
Let $\bar \X = \bar \P \bar \P^T$, where $\bar \X$ is feasible for~\eqref{eq:A1}. Then $\bar \X$ is a local minimizer of~\eqref{eq:A1} if and only if $\bar \P$ is a local minimizer of~\eqref{eq:A2}.
\end{theorem}
\begin{proof}
We follow ideas from both \cite{BurerMonteiro2005} and \cite{Aravkin2014}. From Lemma 2.1 in \cite{BurerMonteiro2005}, if both $\P$ and $K$ are $\dimVar \times \rankVar$ matrices, then $\P\P^T = KK^T$ if and only if $\P=QK$ for some orthogonal matrix $Q\in\R^{\rankVar \times \rankVar}$. The objective in \eqref{eq:A2} depends only on $\P\P^T$, so it is clear that $\P$ is a local minimizer of \eqref{eq:A2} if and only if $\P Q$ is a local minimizer for all orthogonal $Q$.

Note that $g$ defined by $\P \mapsto f(\P\P^T)$ is also lsc since it is the composition of $f$ with the continuous function $\P \mapsto \P\P^T$. We require the functions to be lsc so that the notion of ``local minimizer'' is well-defined.

Suppose $\bar \X$ is a local minimizer of~\eqref{eq:A1}, i.e., for some $\epsilon>0$ there is no better feasible point, and factor it as $\bar X = \bar \P \bar \P^T$ for some $\dimVar \times \rankVar$ matrix $\bar \P$. Then we claim $\bar \P$ is a local minimizer of \eqref{eq:A2}. If it is not, then there exists a sequence $\P_n \rightarrow \bar \P$ with $g(\P_n) < g(\bar P)$ for all $n$. By continuity of the map  $\P \mapsto \P\P^T$, there is some $n$ large enough such that $\X_n \defeq \P_n\P_n^T$ is within $\epsilon$ of $\bar \X$, with $f(\X_n)<f(\bar \X)$, contradicting the local optimality of $\bar \X$.

We prove the other direction using contrapositive. Suppose $\bar \X$ is not a local minimizer of \eqref{eq:A1}, so there is a sequence $\X_n \rightarrow \bar \X$ with $f(\X_n) < f(\bar \X)$. Factor each $\X_n$ as $\X_n = \P_n \P_n^T$, and observe that it is not true that $\P_n$ converges. However, $(\X_n)$ converges and hence is bounded, thus $(\P_n)$ is bounded as well (for the spectral norm, $\|\X_n\|=\|\P_n\|^2$, and over the space of matrices, all norms are equivalent). Since $\P_n$ are in a finite dimensional space, so the Bolzano-Weierstrass theorem guarantees that there is a sub-sequence of $(\P_n)$ that converges. Let the limit of the sub-sequence be $\bar \P = \lim_{k\rightarrow \infty} \P_{n_k}$, and note 
 $\bar \P \bar\P^T = \bar\X$ since $\X_{n_k} \rightarrow \bar X$.
 Then $g(\P_{n_k}) = f(\X_{n_k}) < f(\bar \X) = g(\bar \P)$, so $\bar \P$ is not a local solution. It also follows from the first paragraph of the proof that there can not be another local solution $\tilde \P$ that also satisfies $\tilde \P \tilde \P^T = \bar \X$.
\end{proof}
\begin{remark}
We recover the settings of both \cite[Thm.\ 4.1]{Aravkin2014} and \cite{BurerMonteiro2005}, since allowing $f$ to be extended valued and lsc encompasses constraints.
\end{remark}
\begin{remark}
\cite[Cor.\ 4.2]{Aravkin2014} can be clarified to state that the two problems listed there are ``equivalent'' in the sense that they share the same local minimizers (as well as global), using similar continuity arguments on the mapping $\mathcal{R}$ and its adjoint. Furthermore, since the constraints are compact, solutions exist in both formulations.
\end{remark}

\end{document}